\pdfoutput=1
\documentclass[letterpaper]{amsart}

\usepackage[utf8]{inputenc}
\usepackage{lmodern}
\usepackage[T1]{fontenc}
\usepackage{amssymb}
\usepackage{stmaryrd}
\usepackage{enumitem}
\usepackage{mathtools}
\usepackage{xr}

\newif\ifarxiv
\arxivtrue
\ifarxiv
\usepackage{tikz-cd}
\usepackage[pdfusetitle]{hyperref}
\else

\usepackage{tikz-cd}
\usepackage[pdfusetitle,dvipdfmx]{hyperref}
\fi
\tikzcdset{arrow style=Latin Modern}
\usepackage{cleveref}

\def\arxiv#1{\href{http://arxiv.org/abs/#1}{\texttt{arXiv:#1}}}
\def\doi#1{\href{http://doi.org/#1}{doi:#1}}

\mathtoolsset{mathic}
\DeclareMathAlphabet\mathbfit{OML}{cmm}{b}{it}

\newlist{enumarabic}{enumerate}{1}
\setlist[enumarabic]{font=\normalfont,label=(\arabic*),leftmargin=0.3in}
\newlist{enumroman}{enumerate}{1}
\setlist[enumroman]{font=\normalfont,label=(\roman*),leftmargin=0.3in}

\let\oldFootnote\footnote
\newcommand\nextToken\relax

\renewcommand\footnote[1]{%
    \oldFootnote{#1}\futurelet\nextToken\isFootnote}

\newcommand\isFootnote{%
    \ifx\footnote\nextToken\textsuperscript{,}\fi}

\def\N{\mathbb N}
\def\Z{\mathbb Z}
\def\Q{\mathbb Q}
\def\R{\mathbb R}
\def\C{\mathbb C}

\def\ZZ{\mathcal Z}

\let\phi\varphi
\let\epsilon\varepsilon
\let\emptyset\varnothing

\def\deg#1{|#1|}

\DeclareMathOperator{\Tor}{Tor}

\DeclareMathOperator{\Hom}{Hom}

\DeclareMathOperator*{\colim}{colim}
\DeclareMathOperator{\lin}{lin}

\def\AW{AW}

\def\AWu#1{\AW_{\mkern -1mu #1}}

\def\cupone{\mathbin{\cup_1}}
\def\cuptwo{\mathbin{\cup_2}}

\def\kk{\Bbbk}
\def\kkZ{\Z}

\def\Kl{\mathbf K}

\def\tildecc{\tilde c}

\def\ffbar{f}

\def\susp{\mathrm{s}}
\def\desusp{\susp^{-1}}

\def\Simp{\mathcal S}

\let\transpp\transp

\let\newterm\emph

\def\MM{\mathbf{M}}

\newcommand*\xbar[1]{%
   \hbox{%
     \vbox{%
       \hrule height 0.35pt 
       \kern0.35ex
       \hbox{%
         \kern-0.15em
         \ensuremath{#1}%
         \kern-0.15em
       }%
     }%
   }%
}

\def\EE{\mathbf{E}}

\def\iter#1#2{#1^{[#2]}}
\def\BB{\mathbf{B}}

\def\PsiSigma{\Psi_{\Sigma}}
\def\tildePsiSigma{\tilde\Psi_{\Sigma}}

\def\refhomog#1{\ref*{h@#1}}

\makeatletter
\def\eqrefhomog#1{\textup{\tagform@{\ref*{h@#1}}}}
\makeatother
\def\citehomog#1{\cite[#1]{Franz:2019}}

\def\refgersten#1{\ref*{g@#1}}
\makeatletter
\def\eqrefgersten#1{\textup{\tagform@{\ref*{g@#1}}}}
\makeatother
\def\citegersten#1{\cite[#1]{Franz:2018a}}

\theoremstyle{plain}
\newtheorem{theorem}{Theorem}[section]
\newtheorem{proposition}[theorem]{Proposition}
\newtheorem{lemma}[theorem]{Lemma}
\newtheorem{corollary}[theorem]{Corollary}

\theoremstyle{definition}
\newtheorem{remark}[theorem]{Remark}
\newtheorem{example}[theorem]{Example}

\theoremstyle{remark}
\newtheorem*{acknowledgements}{Acknowledgements}

\numberwithin{equation}{section}

\def\cf{\emph{cf.}}

\def\RP{\mathbb{RP}}

\def\HK{H_{K}}
\def\HL{H_{L}}

\def\Kl{\mathbf K}
\def\mmm{\mathfrak{m}}

\def\barsim#1{[#1]}
\def\bigbarsim#1{\bigl[#1\bigr]}
\def\barbarsim#1{\llbracket#1\rrbracket}

\def\Bar#1#2{\BB(\kk,#2,#1)}
\def\BarEl#1#2{#2\otimes#1}

\def\id{1}

\def\tildeq{\tilde q}

\def\hatalpha{\skew{-1}\hat\alpha}

\def\aa{\mathbfit{a}}
\def\bb{\mathbfit{b}}
\let\topfont\mathcal
\let\algfont\mathbfit

\def\TT{\topfont{T}}
\def\GG{\topfont{G}}
\def\KK{\topfont{K}}
\def\LL{\topfont{L}}
\def\XX{\topfont{X}}
\def\YY{\topfont{Y}}
\def\ZZ{\topfont{Z}}
\def\SS{\topfont{S}}
\def\DD{\topfont{D}}
\def\EEGG{\topfont{E G}}
\def\TTT{\algfont{T}}
\def\KKK{\algfont{K}}
\def\LLL{\algfont{L}}
\def\XXX{\algfont{X}}
\def\YYY{\algfont{Y}}
\def\ZZZ{\algfont{Z}}

\def\kkV{\kk[V]}

\def\coord#1#2{#2^{#1}}

\def\SimpDJ{DJ}
\def\SimpSOne{S^{1}}
\def\Simp{S}

\def\PolyProd#1#2#3{(#2,#3)^{#1}}
\def\bigPolyProd#1#2#3{\bigl(#2,#3\bigr)^{#1}}

\def\DDDD{\mathbb D}

\def\APL{A_{\mathrm{PL}}}

\begin{document}

\title[Cohomology rings of smooth toric varieties]
  {The cohomology rings of smooth toric varieties and quotients of moment-angle complexes}
\author{Matthias Franz}
\thanks{The author was supported by an NSERC Discovery Grant.}
\address{Department of Mathematics, University of Western Ontario, London, Ont.\ N6A\;5B7, Canada}
\email{mfranz@uwo.ca}

\dedicatory{Dedicated to Volker Puppe on the occasion of his 80th birthday}
      

\begin{abstract}
  Partial quotients of moment-angle complexes are topological analogues of
  smooth, not necessarily compact toric varieties. In 1998, Buchstaber and Panov
  proposed a formula for the cohomology ring of such a partial quotient
  in terms of a torsion product involving the corresponding Stanley--Reisner ring.
  We show that their formula gives the correct cup product
  if \(2\) is invertible in the chosen coefficient ring,
  but not in general.
  We rectify this by defining an explicit deformation
  of the canonical multiplication on the torsion product.
\end{abstract}

\maketitle

\section{Introduction}

Let \(\Sigma\) be a simplicial complex on the finite vertex set~\(V\).
The moment-angle complex associated to~\(\Sigma\) is the space
\begin{align}
  \label{eq:intro-def-ma}
  \ZZ_{\Sigma}(D^{2},S^{1}) &= \bigcup_{\smash{\sigma\in\Sigma}} \ZZ_{\sigma}(D^{2},S^{1}) \subset (D^{2})^{V}, \\
\shortintertext{where \(D^{2}\),~\(S^{1}\subset\C\) denote the unit disc and circle, and the exponents in}
  \ZZ_{\sigma}(D^{2},S^{1}) &= (D^{2})^{\sigma}\times (S^{1})^{V\setminus\sigma}
\end{align}
the index sets for the Cartesian products.
The compact torus~\(T=(S^{1})^{V}\) acts on~\(Z_{\Sigma}=\ZZ_{\Sigma}(D^{2},S^{1})\)
in the canonical way.
Moment-angle complexes are compact analogues of
complements of complex coordinate subspace arrangements,
which are exactly the toric subvarieties of affine space.
Like a simplicial fan, the simplicial complex~\(\Sigma\)
can be identified with a subfan of the positive orthant in~\(\R^{V}\),
defining (with the obvious notation)
a toric subvariety~\(\ZZZ_{\Sigma} = \ZZ_{\Sigma}(\C,\C^{\times}) \subset \C^{V}\).
The canonical inclusion~\(Z_{\Sigma}\hookrightarrow\ZZZ_{\Sigma}\)
then is a \(T\)-equivariant strong deformation retract.

Buchstaber--Panov call the quotient~\(X_{\Sigma}=Z_{\Sigma}/K\) by a freely acting subtorus
a \newterm{partial quotient} \cite[Sec.~7.5]{BuchstaberPanov:2002}. We do so
more generally for any freely acting closed subgroup~\(K\) of~\(T\) with quotient~\(L\).
The motivation for this construction comes again from toric geometry:
Any smooth toric variety~\(\XXX_{\Sigma}\) is of the form~\(\ZZZ_{\Sigma}/\KKK\)
for an algebraic subgroup~\(\KKK\) of~\(\TTT=(\C^{\times})^{V}\)
with quotient~\(\LLL\);
this presentation is often called the Cox construction.
As before, the inclusion~%
\( 
  X_{\Sigma} \hookrightarrow \XXX_{\Sigma}
\) 
is an \(L\)-equivariant strong deformation retract. 

Bifet--De\,Concini--Procesi~\cite[Thm.~8]{BifetDeConciniProcesi:1990}
have identified the \(L\)-equivariant cohomology of~\(\XXX_{\Sigma}\)
with the evenly graded Stanley--Reisner algebra~\(\kk[\Sigma]\) of~\(\Sigma\);
the module structure over~\(H^{*}(BL;\kk)\)
is induced by the projection~\(T\to L\).
This holds for coefficients in any principal ideal domain~\(\kk\) and
generalizes directly to partial quotients.

By the Eilenberg--Moore theorem,
there is a spectral sequence of algebras
converging to~\(H^{*}(X_{\Sigma};\kk)=H^{*}(\XXX_{\Sigma};\kk)\) with second page
\begin{equation}
  \label{eq:intro:E2}
  E_{2} = \Tor^{*}_{H^{*}(BL;\kk)}(\kk,\kk[\Sigma]).
\end{equation}
Recall that the torsion product of graded commutative algebras
is bigraded and carries a canonical product respecting bidegrees.
Moreover, \(\Tor\) is graded negatively in the cohomological setting.
If \(\kk[\Sigma]\) is free over~\(H^{*}(BL;\kk)\), then \eqref{eq:intro:E2} gives
the Jurkiewicz--Danilov formula~%
\(H^{*}(\XXX_{\Sigma};\kk)\cong\kk\otimes_{H^{*}(BL;\kk)}\kk[\Sigma]\),
\cf~\cite[Thm.~5.3.1]{BuchstaberPanov:2015}.

In~1998, Buchstaber--Panov claimed that for connected~\(K\)
the above spectral sequence collapses in the best possible way,
namely that \(\HL^{*}(X)\) is isomorphic to~\eqref{eq:intro:E2}
as a graded algebra,
\begin{equation}
  \label{eq:intro:iso-tor}
  H^{*}(X_{\Sigma};\kk) \cong \Tor^{*}_{H^{*}(BL;\kk)}(\kk,\kk[\Sigma]),
\end{equation}
with the cohomological degree corresponding to the total degree
in the torsion product.
It seems that this statement first appeared in~\cite[Thm.~5]{BuchstaberPanov:1998}
for moment-angle complexes and field coefficients
and in~\cite[Thm.~4.13]{BuchstaberPanov:1999} for the general case.
It has since propagated to several other publications,
including the textbook~\cite[Thms.~7.6~\&~7.37]{BuchstaberPanov:2002}.
Buchstaber--Panov's proof, however, was found to be incorrect in~2003.

An isomorphism of \(\kk\)-modules of the form~\eqref{eq:intro:iso-tor}
was established by the author
in his doctoral dissertation~\cite[Thm.~3.3.2]{Franz:2001},~\cite[Thm.~1.2]{Franz:2006}
for any~\(K\subset T\). The product structure turned out to be more elusive
(and was also incorrectly stated in~\cite{Franz:2001}
and conjectured in~\cite{Franz:2006}).
In the case of moment-angle complexes, the Buchstaber--Panov product formula was
confirmed by the author~\cite[Thm.~1.3]{Franz:2003a}. A different proof appeared
shortly afterwards in the expanded Russian edition~\cite[Thm.~8.6]{BuchstaberPanov:2004}
of~\cite{BuchstaberPanov:2002} as well as in the
paper~\cite{BaskakovBuchstaberPanov:2004} by Baskakov--Buchstaber--Panov.

In~2015, Panov published a short note~\cite{Panov:2015}
in which he sketched a new construction
of a multiplicative isomorphism
for any subtorus~\(K\subset T\) and any commutative ring~\(\kk\).
There is no way, however,
to make the isomorphism~\eqref{eq:intro:iso-tor} compatible with products in general.

\begin{example}
  \label{ex:intro}
  \def\Ztwo{\Z_{2}}
  Consider the smooth toric variety~\(\XXX=\ZZZ/\KKK\)
  where \(\ZZZ=\C^{2}\setminus\{0\}\) and~\(\KKK=\{\pm1\}\).
  For~\(\kk=\Z_{2}\),
  the second page of the Eilenberg--Moore spectral sequence has the form
  \begin{equation}
    \begin{array}{ccc|c}
      & \Ztwo & & 4 \\
      & \Ztwo & \Ztwo & 2 \\
      & & \Ztwo & 0 \\
      \hline
      -2 & -1 & 0 & \mathrlap{\;\;\;.}
    \end{array}
  \end{equation}
  Since the product on the \(\Tor\)~term respects bidegrees,
  the square of the generator in bidegree~\((-1,2)\) vanishes.
  On the other hand, the corresponding element in~\(H^{1}(\XXX;\Ztwo)\)
  does not square to zero as \(\XXX\) is homotopy-equivalent to~\(\RP^{3}\).
\end{example}

In \Cref{sec:examples} we will generalize this example
and show that even for quotients by subtori
a multiplicative isomorphism of the form~\eqref{eq:intro:iso-tor}
may fail to exist.
These counterexamples work over~\(\kk=\Z_{2}\) and~\(\kk=\Z\).

Our main result says that the prime~\(2\) is the only one that poses problems.

\begin{theorem}
  \label{thm:intro:main-mod2}
  Assume that \(2\) is invertible in~\(\kk\).
  Then there is an isomorphism of graded \(\kk\)-algebras
  \begin{equation*}
    H^{*}(\XXX_{\Sigma};\kk) \cong \Tor^{*}_{H^{*}(BL;\kk)}(\kk,\kk[\Sigma]).
  \end{equation*}
\end{theorem}

\goodbreak

In order to describe the product in~\(H^{*}(\XXX_{\Sigma};\kk)\) and~\(H^{*}(X_{\Sigma};\kk)\)
for any coefficient ring~\(\kk\),
we consider a deformation of the canonical multiplication
on the torsion product~\eqref{eq:intro:E2}.
We need to introduce some notation to state the formula.

For~\(v\in V\), let \(x_{v}\in H_{1}(L;\Z)\) be the image of the \(v\)-th canonical basis vector
of~\(H_{1}(T;\Z)\) under the projection~\(T\to L\). In the context of toric varieties,
\(x_{v}\) is the minimal integral representative of the ray corresponding to~\(v\).
Recall that a differential graded algebra
computing the torsion product~\eqref{eq:intro:E2}
together with its canonical multiplication is the tensor product
\begin{equation}
  \label{eq:intro:Tor-complex}
  \Kl_{\Sigma} = H^{*}(L;\kk)\otimes_{\kk} \kk[\Sigma]
\end{equation}
with componentwise product and Koszul differential determined by
\begin{equation}
  \label{eq:intro:koszul-diff}
  df = 0
  \qquad\text{and}\qquad
  d\alpha = -\sum_{v\in V}\alpha(x_{v})\, t_{v}\,,
\end{equation}
where the~\(t_{v}\in H^{2}(BT)\) are the canonical generators of the polynomial algebra.

Let us write \(\coord{1}{x_{v}}\),~\dots,~\(\coord{n}{x_{v}}\in\Z\)
for the coordinates of~\(x_{v}\) with respect to some basis for~\(H_{1}(L;\Z)\).
We define linear elements \(q_{ij}\in\kkZ[\Sigma]\) for~\(i\ge j\) by
\begin{equation}
  q_{ii} = \sum_{v\in V} \frac{\coord{i}{x_{v}}(\coord{i}{x_{v}}-1)}{2}\,t_{v},
  \qquad
  q_{ij} = \sum_{v\in V}\coord{i}{x_{v}}x_{v}^{j}\,t_{v}
  \quad\text{if~\(i>j\),}
\end{equation}
noting that the coefficients in the left-hand formula are indeed integers.
The twisted product~\(*\) on the complex~\(\Kl_{\Sigma}\) has the form
\begin{equation}
  (\alpha f)*(\beta g) = \alpha\beta\,fg
  + \sum_{i\ge j} \coord{i}{\alpha}\coord{j}{\beta} \, q_{ij} fg
\end{equation}
for~\(\alpha\),~\(\beta\in H^{1}(L;\kk)\)
having coordinates~\(\coord{i}{\alpha}\),~\(\coord{j}{\beta}\in\kk\) with respect
to the dual basis.
For exterior algebra elements of higher degrees there are additional twisting terms,
see \Cref{sec:twisted-prod} for the precise formula.

\begin{theorem}
  \label{thm:intro:main-general}
  Let \(\kk\) be a principal ideal domain.
  There is an isomorphism of graded algebras
  \begin{equation*}
    H^{*}(\XXX_{\Sigma};\kk) \cong \Tor^{*}_{H^{*}(BL;\kk)}(\kk,\kk[\Sigma])
  \end{equation*}
  where the multiplication on the right-hand side
  is induced by the \(*\)-product on~\(\Kl_{\Sigma}\).
\end{theorem}

The differential graded algebra of singular cochains
on the Borel construction of a partial quotient
or a smooth toric variety is formal over any coefficient ring~\(\kk\).
This also appeared in the author's doctoral dissertation~\cite[Thm.~3.3.2]{Franz:2001}, \cite[Thm.~1.1]{Franz:2003a}, \cite[Thm.~1.4]{Franz:2006},
and was independently obtained by Notbohm--Ray~\cite[Thm.~4.8]{NotbohmRay:2005} shortly afterwards.
In the companion paper~\cite{Franz:2018a} we extend this to
homotopy Gerstenhaber algebras, \cf~\Cref{thm:formality-DJ-short} below.
The proofs of Theorems~\ref{thm:intro:main-mod2} and~\ref{thm:intro:main-general}
crucially depend on this significantly stronger formality result.

For these two theorems
to hold, it is actually not necessary for the subgroup
to act freely on the Cox construction or the moment-angle complex.
It suffices that all isotropy groups are finite
and their orders invertible in~\(\kk\),
see Corollaries~\ref{thm:free-general} and~\ref{thm:free-2}.
For toric varieties, this is equivalent to the \(\kk\)-regularity of the fan.
Without this assumption,
Theorems~\ref{thm:intro:main-mod2} and~\ref{thm:intro:main-general} 
remain valid with \(H^{*}(\XXX_{\Sigma};\kk)=H^{*}(\ZZZ_{\Sigma}/\KKK;\kk)\)
replaced by the equivariant cohomology~\(\HK^{*}(\ZZZ_{\Sigma};\kk)\),
see Theorems~\ref{thm:iso-tor-general} and~\ref{thm:iso-tor-2}.
Moreover, all results hold for moment-angle complexes defined from simplicial posets
instead of simplicial complexes.

\begin{acknowledgements}
  I am indebted to Suyoung Choi for resuscitating my interest in this question,
  to Volker Puppe for suggesting the short proof of \Cref{thm:free-quotient},
  to Xin Fu for a critical reading of the proof of \Cref{thm:iso-tor-general}
  and to Derek Krepski for explanations about stacks.

  This work completes a project started in my dissertation~\cite{Franz:2001}
  under the supervision of Volker Puppe. To thank him for the many things
  he taught me and for our long and fruitful collaboration that I have very much enjoyed,
  I dedicate this paper to him.
\end{acknowledgements}

\section{Bar constructions}
\label{sec:bar}

Throughout this paper,
the letter~\(\kk\) denotes a principal ideal domain.
Unless specified otherwise, all tensor products as well as
all (co)chain complexes and all (co)homologies are over~\(\kk\).
We write the identity map on a \(\kk\)-module~\(M\) as~\(\id_{M}\).

We call a differential graded algebra (dga)~\(A\)
connected if it is \(\N\)-graded with~\(A^{0}\cong\kk\); if additionally \(A^{1}=0\),
then it is \(1\)-connected. The same applies to dg~coalgebras (dgcs).
We denote the multiplication map of a dga~\(A\) by~\(\mu_{A}\)
and the diagonal of a dgc~\(C\) by~\(\Delta_{C}\).
Augmentations are written as \(\epsilon\).
A (graded) commutative dga (cdga)
is a dga~\(A\) such that \(ab=(-1)^{\deg{a}\deg{b}}\,ba\) for all~\(a\),~\(b\in A\).

Let \(A\) be an augmented dga with differential of degree~\(+1\).
As a graded \(\kk\)-module, the (reduced) bar construction~\(\BB A\) of~\(A\) is
the direct sum of the pieces \(\BB_{k}A=(\desusp\bar A)^{\otimes k}\)
for~\(k\ge0\), where \(\desusp\bar A\) is the desuspension
of the augmentation ideal~\(\bar A=\ker\epsilon\).
We say that an element~\([a_{1}|\dots|a_{k}]\in \BB_{k}A\) has \newterm{length}~\(k\).
Its differential is
\begin{align}
  \label{eq:bar-differential}
  d\,[a_{1}|\cdots|a_{k}] &= -\sum_{i=1}^{k}(-1)^{\epsilon_{i-1}}\,[a_{1}|\cdots|a_{i-1}|da_{i}|a_{i+1}|\cdots|a_{k}] \\
  \notag &\qquad\qquad +\sum_{i=1}^{\smash{k-1}} (-1)^{\epsilon_{i}}\,[a_{1}|\cdots|a_{i}a_{i+1}|\cdots|a_{k}]
\end{align}
where
\(
  \epsilon_{i} = \deg{a_{1}} + \dots + \deg{a_{i}} - i
\),
and its image under the diagonal is
\begin{equation}
  \Delta_{\BB A} [a_{1}|\dots|a_{k}] = \sum_{i=0}^{k}[a_{1}|\dots|a_{i}]\otimes[a_{i+1}|\dots|a_{k}].
\end{equation}
With these definitions
\(\BB A\) becomes a dgc, which is connected if \(A\) is \(1\)-connected.
The canonical twisting cochain~\(t_{A}\colon \BB A\to A\) is the composition
of the canonical projection onto~\(\BB_{1}A=\desusp\bar A\)
with the map~\(\desusp\bar A\to\bar A\hookrightarrow A\).

Given a morphism of dgas~\(g\colon A\to A'\), we can form the one-sided bar construction
as the twisted tensor product
\begin{equation}
  \Bar{A'}{A} = \BB A \otimes_{t_{A}} A'.
\end{equation}
This is the tensor product~\(\BB A\otimes A'\) with the differential
\begin{equation}
  \label{eq:def-d-BAA}
  d = d_{\BB A}\otimes\id_{A'} + \id_{\BB A}\otimes d_{A'}
  - (\id_{\BB A}\otimes\mu_{A'})(1\otimes g\,t_{A}\otimes1)(\Delta_{\BB A}\otimes\id_{A'}) .
\end{equation}
Explicitly, the differential is given by
\begin{equation}
  \label{eq:def-d-bar}
  d(\aa\otimes a)
  = d\aa\otimes a
  +(-1)^{\deg{\aa}}\,\aa\otimes da \\
  -(-1)^{\epsilon_{k-1}}
  \,[a_{1}|\dots|a_{k-1}]\otimes g(a_{k})\,a
\end{equation}
for~\(\aa=[a_{1}|\dots|a_{k}]\in \BB A\), \(a\in A\) and the same sign exponent as before.
The last summand is read as~\(0\) for~\(k=0\).
We write the cohomology of the one-sided bar construction as
the differential torsion product
\begin{equation}
  \Tor_{A}(\kk,A') = H^{*}(\Bar{A'}{A}).
\end{equation}
From now on, we will omit the map~\(g\) from our notation.

Given a morphism of dgas~\(f\colon A'\to A''\), we can form
the one-sided bar construction~\(\Bar{A''}{A}\) in the obvious way.
We denote it by~\(\BB_{f}(\kk,A,A'')\) to remind ourselves
that the differential depends on~\(f\). There is a canonical chain map
\begin{equation}
  \Psi_{f}\colon \Bar{A'}{A}\to \BB_{f}(\kk,A,A''),
  \quad
  \aa\otimes a\mapsto \aa\otimes f(a).
\end{equation}

\begin{lemma}
  \label{thm:Psi-quiso}
  Assume that \(A\) is \(1\)-connected and torsion-free over~\(\kk\) and that \(A'\) is bounded below.
  Then \(\Psi_{f}\) is a quasi-iso\-mor\-phism if \(f\) is so.
\end{lemma}

\begin{proof}
  This is a standard spectral sequence argument.
  We assume \(A\) to be \(1\)-con\-nected to ensure that
  the filtration of~\(\BB A\) is finite in each degree.
\end{proof}

We recall the definition of a \newterm{homotopy Gerstenhaber algebra} (homotopy G-al\-ge\-bra, \newterm{hga}),
due to Voronov--Gerstenhaber~\cite[\S 8]{VoronovGerstenhaber:1995},
see also~\citehomog{Sec.~\refhomog{sec:def-hga}}.
An hga is an augmented dga with differential of degree~\(+1\)
together with a multiplication
\begin{equation}
  \label{eq:hga-product-bar}
  \mu_{\BB A}\colon \BB A \otimes \BB A \to \BB A
\end{equation}
turning the bar construction into a dg~bialgebra, that is,
into a dga such that the diagonal~\(\Delta_{\BB A}\) is a morphism of dgas.
There is an additional requirement on
the twisting cochain~\(t\colon \BB A\otimes \BB A\to A\) determining~\(\mu_{\BB A}\).
If we write the components of~\(t\) as
\begin{equation}
  \label{eq:hga-t}
  \EE_{kl}\colon \BB_{k}A \otimes \BB_{l}A
  \to \BB_{1}A = \desusp\bar A \stackrel{\susp}{\longrightarrow}\bar A \hookrightarrow A
\end{equation}
with~\(k\),~\(l\ge0\), then we require
that both~\(\EE_{0}\coloneqq \EE_{10}\) and~\(\EE_{01}\) are the map
\begin{equation}
  \BB_{1}A\otimes \BB_{0}A = \BB_{0}A\otimes \BB_{1}A
  = \desusp\bar A \stackrel{\susp}{\longrightarrow}\bar A \hookrightarrow A
\end{equation}
and that otherwise the only non-zero components are the maps~\(\EE_{l}\coloneqq \EE_{1l}\).
We write the product as~\(\aa\circ\bb\in \BB A\).
Any cdga is canonically an hga by setting \(\EE_{k}=0\) for all~\(k\ge1\).

Conversely, the maps~\(\EE_{k}\) determine the product~\eqref{eq:hga-product-bar} by the general procedure
of recovering a dgc map~\(C\to\BB A\) from the associated twisting cochain~\(C\to A\).
The formula given in~\citehomog{eq.~\eqrefhomog{eq:shm-to-dgc}} for the case where \(C\) is a bar construction carries over.
By way of example, for elements~\(a_{1}\),~\(a_{2}\),~\(b_{1}\),~\(b_{2}\in A\) of even degrees one has
\begin{align}
  \MoveEqLeft{[a_{1}|a_{2}]\circ[b_{1}|b_{2}] = {}} \\
  \notag & \phantom{{}+{}} [a_{1}|a_{2}|b_{1}|b_{2}] - [a_{1}|b_{1}|a_{2}|b_{2}] + [a_{1}|b_{1}|b_{2}|a_{2}] \\
  \notag & {} + [b_{1|}a_{1}|a_{2}|b_{2}] - [b_{1}|a_{1}|b_{2}|a_{2}] + [b_{1}|b_{2}|a_{1}|a_{2}] \\
  \notag & {} + [a_{1}|\EE_{1}(a_{2};b_{1})|b_{2}] + [a_{1}|\EE_{2}(a_{2};b_{1},b_{2})] - [\EE_{1}(a_{1};b_{1})|a_{2}|b_{2}] \\
  \notag & {} - [a_{1}|b_{1}|\EE_{1}(a_{2};b_{2})] - [\EE_{1}(a_{1};b_{1})|\EE_{1}(a_{2};b_{2})] + [\EE_{1}(a_{1};b_{1})|b_{2}|a_{2}] \\
  \notag & {} + [\EE_{2}(a_{1};b_{1},b_{2})|a_{2}] + [b_{1}|a_{1}|\EE_{1}(a_{2};b_{2})] - [b_{1}|\EE_{1}(a_{1};b_{2})|a_{2}].
\end{align}
This illustrates the general pattern: For each term in the shuffle product (the first two lines above),
additional terms with the same sign appear where each \(a\)-variable can form an \(\EE\)-term with one or more \(b\)-variables immediately following it.
Hence a product~\([a_{1}|\dots|a_{k}]\circ[b_{1}|\dots|b_{l}]\)
has components of length at most~\(k+l\).
The component of length~\(k+l\) is the usual shuffle product,
the one of length~\(1\)  (occurring only for~\(k=1\) as well as for~\(k=0\) and~\(l=1\)) is given by the maps~\(\EE_{l}\) plus~\(\EE_{01}\),
and the one of length~\(0\) vanishes unless \(k=l=0\).
In particular, one has
\begin{equation}
  \label{eq:prod-bar-1-1}
  [a_{1}]\circ[b_{1}]
  = [a_{1}|b_{1}]+(-1)^{(\deg{a_{1}}-1)(\deg{b_{1}}-1)}\,[b_{1}|a_{1}]
  +\bigl[\EE_{1}([a_{1}],[b_{1}])\bigr].
\end{equation}

The fact that \(\mu_{\BB A}\) is the multiplication of a dga can be expressed by identities involving the maps~\(\EE_{k}\)
or, equivalently, the corresponding suspended operations \(E_{k}\colon A\otimes A^{\otimes k}\to A\).
See~\citehomog{Sec.~\refhomog{sec:def-hga}} for such a list using our sign conventions.

A morphism of hgas~\(A\to A'\) is a morphism of dgas that is compatible with the operations~\(\EE_{k}\)
in the obvious way.

\begin{proposition}[Kadeishvili--Saneblidze]
  \label{thm:def-prod-bar}
  Let \(A\to A'\) be a morphism of hgas.
  Then \(\Bar{A'}{A}\) is naturally an augmented dga
  with unit~\(1_{\BB A}\otimes1_{A'}\), augmentation~\(\epsilon_{\BB A}\otimes\epsilon_{A'}\)
  and product of the form
  \begin{align*}
    (\BarEl{a}{\aa}) \circ (\BarEl{b}{\bb})
    &= (-1)^{\deg{a}\deg{\bb}}\, \BarEl{ab}{(\aa\circ\bb)} \\
    &\qquad + \sum_{i=1}^{l}\pm
    \BarEl{\EE_{i}([a],[b_{l-i+1}|\dots|b_{l}])\,b}{(\aa\circ[b_{1}|\dots|b_{l-i}])}
  \end{align*}
  for~\(\aa\),~\(\bb=[b_{1}|\dots|b_{l}]\in \BB A\) and~\(a\),~\(b\in A'\).
\end{proposition}

\begin{proof}
  This is stated in~\cite[Thm.~7.1]{KadeishviliSaneblidze:2005}
  in the form
  \begin{equation}
  \begin{split}
    (a\otimes\aa) \circ (b\otimes \bb)
    &= (-1)^{\deg\aa\deg{b}}\, ab \otimes (\aa\circ\bb) \\
    &\qquad + \sum_{i=1}^{k} \pm a\,\EE_{i}([a_{1}|\dots|a_{i}],[b])\otimes
    \bigl([a_{i+1}|\dots|a_{k}]\circ\bb\bigr)
  \end{split}
  \end{equation}
  with specified signs (which are corrected in~\cite[Rem.~2]{Saneblidze:2009}).
  Note that in~\cite{KadeishviliSaneblidze:2005} a homotopy Gerstenhaber structure
  is defined to have operations~\(\EE_{k1}\colon \BB_{k}A\otimes \BB_{1}A\to A\), which differs
  from our convention. One can change between the two by transposing the factors
  in the product, \cf~\citegersten{Rem.~\refgersten{rem:other-hga-structure}}.

  More importantly, Kadeishvili--Saneblidze state this
  only for simply connected hgas. However, it follows formally
  from the properties enjoyed by the operations~\(\EE_{k}\) of any hga.
\end{proof}

Note that \(\BB A\) and~\(A'\) are canonically subalgebras of~\(\Bar{A'}{A}\).
By abuse of notation, we write \(\aa\) and~\(a\)
instead of~\(\BarEl{1_{A'}}{\aa}\),~\(\BarEl{a}{1_{\BB A}}\in\Bar{A'}{A}\).

\section{Simplicial sets}

It will be convenient to work with simplicial sets.
We use \cite{May:1968} as our reference.

\subsection{Generalities}
\label{sec:simp-general}

We refer to~\cite[\S\S 1,\,17]{May:1968} for the definition of simplicial sets,
simplicial groups and Kan complexes.
To distinguish simplicial sets from topological spaces,
we write the former as~\(X\),~\(Y\),~\(Z\) and the latter in the form~\(\XX\),~\(\YY\),~\(\ZZ\).
Complex algebraic varieties are still written as~\(\XXX\),~\(\YYY\),~\(\ZZZ\).

Recall that a simplicial set~\(X\) is \newterm{reduced} if it has a single vertex
and \newterm{\(1\)-reduced} if it has a single \(1\)-simplex.
For any simplicial group~\(G\) and any~\(p\ge0\), we write \(1_{p}\in G\) for the simplex
that acts as the identity element of the group of \(p\)-simplices.
For any topological space~\(\XX\) we write \(\Simp(\XX)\) for the associated Kan complex
of singular simplices in~\(\XX\). If \(\GG\) is a topological group, then \(\Simp(\GG)\)
is a simplicial group.

We repeatedly use the fact that a simplicial map between connected Kan complexes
is a homotopy equivalence if and only if it induces isomorphisms
on all homotopy groups, \cf~\cite[Thm.~12.5]{May:1968}.

We write \(C(X)\) for the normalized chain complex of the simplicial set~\(X\)
and \(C^{*}(X)\) for the dga of normalized cochains.
The latter is \(1\)-connected if \(X\) is \(1\)-reduced,
and it has a natural hga structure given by interval cut operations, see~%
\cite[\S 1.6.6]{BergerFresse:2004} as well as~\citehomog{Sec.~\refhomog{sec:cochains-hga}} for our sign conventions.
For~\(\beta\),~\(\gamma\in C^{*}(X)\), the formula
\begin{equation}
  \label{eq:def-cup-one}
\begin{split}
  \beta\cupone\gamma
  &= -\EE_{1}(\desusp\otimes\desusp)(\beta\otimes\gamma)
  = (-1)^{\deg{\beta}-1}\,\EE_{1}([\beta],[\gamma]) \\
  &= -\transpp{\AWu{(1,2,1)}}(\beta\otimes\gamma)
\end{split}  
\end{equation}
defines the usual \(\cupone\)-product,
\cf~\citehomog{eqs.~\eqrefhomog{eq:cupone-d},~\eqrefhomog{eq:hirsch-formula}}.
Here \(\transpp{\AWu{(1,2,1)}}\) denotes the transpose of the interval
cut operation determined by the surjection~\((1,2,1)\).
The \(\cupone\)-product vanishes if one argument is of degree~\(0\).

\subsection{Borel constructions}

Let \(G\) be a simplicial group. A free \(G\)-action on a simplicial set~\(X\)
induces a principal \(G\)-bundle~\(X\to X/G\).
The universal \(G\)-bundle \(EG\to BG\)
is described in~\cite[\S 21]{May:1968}, see also~\citegersten{Sec.~\refgersten{sec:universal-bundles}}.
The canonical basepoints of~\(EG\) and~\(BG\) are denoted by~\(e_{0}\) and~\(b_{0}\), respectively.
We define the \newterm{Borel construction} and the \newterm{\(G\)-equivariant cohomology}
of a simplicial \(G\)-set~\(X\) in the usual way, that is,
\begin{equation}
  X_{G} = EG \times_{G} X = (EG\times X)/G
  \qquad\text{and}\qquad
  H_{G}^{*}(X)=H^{*}(X_{G}).
\end{equation}
The equivariant cohomology of~\(X\)
is an \(H^{*}(BG)\)-algebra via the canonical projection~\(X_{G}\to BG\).

\goodbreak

\begin{lemma} \( \)
  \label{thm:HG-top-sing}
  \begin{enumroman}
  \item
    \label{thm:HG-top-sing-1}
    Let \(G\) be a simplicial group, \(E\) a contractible Kan complex
    with a free \(G\)-action and \(X\) a simplicial \(G\)-set.
    Then there is a homotopy equivalence
    \(E\times_{G} X\to EG\times_{G}X\) whose homotopy class is natural in the given data.
  \item
    \label{thm:HG-top-sing-2}
    Let \(\GG\) be a topological group and \(\XX\) a \(\GG\)-space.
    Then there is a homotopy equivalence~\(\Simp(\EEGG\times_{\GG}\XX)\to E\Simp(\GG)\times_{\Simp(\GG)}\Simp(\XX)\)
    whose homotopy class is natural in~\(\GG\) and~\(\XX\).
    In particular, the topological equivariant cohomology~\(H_{\GG}^{*}(\XX)\)
    and the simplicial equivariant cohomology~\(H_{\Simp(\GG)}^{*}(\Simp(\XX))\)
    are isomorphic.
  \item
    \label{thm:HG-top-sing-3}
    Let \(X\) and~\(X'\) be Kan complexes with actions of the simplicial groups~\(G\)
    and~\(G'\), respectively,
    and let \(X\to X'\) be a map that is  equivariant with respect
    to a map~\(G\to G'\) of simplicial groups.
    If both maps are homotopy equivalences,
    then so is induced map~\(EG\times_{G}X\to EG'\times_{G'}X'\).
    In particular, \(H_{G}^{*}(X)\) and \(H_{G'}^{*}(X')\) are isomorphic.
  \end{enumroman}
\end{lemma}

\begin{proof}
  We start with the first claim. The principal \(G\)-bundle~\(E\to B=E/G\)
  is induced by a map~\(h\colon B\to BG\), unique up to homotopy.
  Comparison of the long exact sequences of homotopy groups \cite[Thm.~7.6]{May:1968}
  for the two bundles shows that \(h\) is a homotopy equivalence.
  An analogous argument for the associated bundles \(E\times_{G}X\) and~\(EG\times_{G}X\)
  completes the argument.

  The second claim now follows because \(\Simp(\EEGG)\) is a contractible
  Kan complex with a free action of~\(\Simp(\GG)\), and
  \(\Simp(\EEGG\times_{\GG}\XX)=\Simp(\EEGG)\times_{\Simp(\GG)}\Simp(\XX)\).

  For the third claim one compares the two Borel constructions
  as in part~\ref{thm:HG-top-sing-1}.
\end{proof}

\begin{proposition}
  \label{thm:dga-bar}
  Let \(G\) be a simplicial group and \(K\lhd G\) a normal subgroup
  such that the quotient~\(L=G/K\) is \(1\)-reduced.
  Moreover, let \(Z\) be a simplicial \(G\)-set,
  and write \(\iota\colon Z_{K}\hookrightarrow Z_{G}\) for the canonical inclusion.
  Then the map
  \begin{align}
    \Bar{C^{*}(Z_{G})}{C^{*}(BL)} &\to C^{*}(Z_{K}), \\
    [\gamma_{1}|\dots|\gamma_{k}]\,\beta &\mapsto
    \begin{cases}
      \iota^{*}(\beta) & \text{if \(k=0\),} \\
      0 & \text{otherwise}
    \end{cases}
  \end{align}
  is a natural quasi-iso\-mor\-phism of dgas.
\end{proposition}

The product on the one-sided bar construction has been defined in \Cref{thm:def-prod-bar}.

\begin{proof}
  It follows from \Cref{thm:HG-top-sing}\,\ref{thm:HG-top-sing-1}
  that the inclusion~\(Z_{K}\hookrightarrow EG\times_{K} Z\) is a homotopy equivalence.
  From the diagram
  \begin{equation}
    \qquad\quad
    \begin{tikzcd}
      EG\times_{K} Z \arrow{d}[left]{\pi} \arrow{r} & EG/K \arrow{d} \arrow{r} & EL \arrow{d} \\
      \mathllap{Z_{G}={}} EG\times_{G} Z \arrow{r} & BG \arrow{r} & BL
    \end{tikzcd}
  \end{equation}
  we infer that \(EG\times_{K} Z\) can be considered as the total space
  of the pull-back of the universal \(L\)-bundle
  along the composition~\(Z_{G}\to BG\to BL\).
  By the Eilenberg--Moore theorem, compare~\citehomog{Prop.~\refhomog{thm:Eilenberg-Moore}}, the map
  \begin{equation}
    \label{eq:bar-ZK}
    \begin{split}
    \Bar{C^{*}(Z_{G})}{C^{*}(BL)} &\to C^{*}(EG\times_{K}Z), \\
    [\gamma_{1}|\dots|\gamma_{k}]\,\beta &\mapsto
    \begin{cases}
      \pi^{*}(\beta) & \text{if \(k=0\),} \\
      0 & \text{otherwise}
    \end{cases}
    \end{split}
  \end{equation}
  is a quasi-iso\-mor\-phism of complexes, hence so is it prolongation to~\(C^{*}(Z_{K})\).

  To conclude, we observe that the map~\eqref{eq:bar-ZK} is multiplicative
  because cochains on~\(BL\) of positive degree pull back to~\(0\) on~\(Z_{K}\).
\end{proof}

\subsection{Tori}
\label{sec:tori}

A \newterm{simplicial torus} is the classifying space~\(T=BN\)
of a lattice~\(N\) of finite rank. Since \(N\) is abelian, \(BN\) is an abelian simplicial group.
Because the simplicial classifying space construction commutes with
Cartesian products, any simplicial torus is the product of \newterm{simplicial circles}~\(B\Z\).
We prefer to work with simplicial tori because their combinatorial nature
permits explicit constructions on the chain level.
The following simple observation allows us to pass from a compact torus~\((S^{1})^{n}\)
or an algebraic torus~\((\C^{\times})^{n}\) to a simplicial torus.

To set the context for it, we note that any simplicial torus~\(T=BN\) is reduced.
The assignment
\begin{equation}
  \label{eq:iso-N-T}
  N\to T,
  \quad
  x \mapsto \barsim{x}
\end{equation}
is a natural bijection between the elements of~\(N\) and the \(1\)-simplices in~\(T\).
As a consequence, we get a natural isomorphism of groups~\(N=\pi_{1}(T)=H_{1}(T;\Z)\).
Under this identification, any element~\(x\in H_{1}(T;\Z)\)
has the tautological representative~\(\barsim{x}\in C_{1}(T;\Z)\).

\begin{lemma}
  \label{thm:simplicial-torus-inclusion}
  Let \(\TT\) be a compact or algebraic torus, and let \(T=B N\) be
  the simplicial torus associated to~\(N=\pi_{1}(\TT)\).
  There is a morphism of simplicial groups~\(T\to\Simp(\TT)\)
  that is natural in~\(\TT\) and a homotopy equivalence.
\end{lemma}

\begin{proof}
  Any~\(a\in N=\pi_{1}(\TT)\) has a canonical representative,
  namely the \(1\)-parameter subgroup~\(\lambda_{a}\) with period~\(1\).
  Therefore, if to the \(n\)-simplex~\(\aa=[a_{1}|\dots|a_{n}]\in B N\)
  we associate the singular \(n\)-simplex~\(\sigma_{\aa}\) in~\(\TT\) given by
  \begin{equation}
    \label{eq:def-1-psg}
    \sigma_{\aa}(t_{0},\dots,t_{n}) =
    \prod_{k=1}^{n} \lambda_{a_{k}}(t_{k}+\dots+t_{n}),
  \end{equation}
  then we get a map of simplicial groups, as one verifies directly
  using the simplicial structure of~\(B N\).
  By construction, this map is natural in~\(\TT\) and an isomorphism on~\(\pi_{1}\),
  hence a homotopy equivalence between the two Eilenberg--Mac\,Lane complexes.
\end{proof}

We now incorporate the \(1\)-reduced space~\(BT=B(BN)\) into our discussion.
In addition to~\eqref{eq:iso-N-T} we have the natural bijection
\begin{equation}
  \label{eq:iso-N-BT}
  N\to BT,
  \quad
  x \mapsto \barbarsim{x}
\end{equation}
between \(N\) (that is, the \(0\)-simplices
of~\(N\), considered as a simplicial group)
and the \(2\)-simplices in~\(BT\).
Here we have written \(\barbarsim{x}\in BT\)
instead of the more systematic notation~\(\bigbarsim{\barsim{x},1_{1},1_{0}}\)
used in~\cite[p.~87]{May:1968}.
This illustrates that we drop redundant entries from simplices in~\(T\) and~\(BT\).
For instance, if we consider \(N\) as a simplicial group,
then the \(1\)-simplex~\(\barsim{x}\in T\) should also be written as~\(\barsim{x,1_{0}}\).

Dually, both the \(1\)-cochains in~\(T\) and the \(2\)-cochains in~\(BT\)
are naturally isomorphic to the \(\kk\)-module of functions~\(N\to\kk\).
The boundaries of a \(2\)-simplex~\(\barsim{x,y}\in T\)
and a \(3\)-simplex~\(\bigbarsim{\barsim{x,y},\barsim{z}}\in BT\)
are
\begin{align}
  d \, \barsim{x,y} &= \barsim{y}-\barsim{x+y}+\barsim{x}, \\
  \label{ed:d-3-simplex}
  d \, \bigbarsim{\barsim{x,y},\barsim{z}} &=
  \barbarsim{z}-\barbarsim{y+z}+\barbarsim{x+y}-\barbarsim{x}.
\end{align}
This implies that \(1\)-cocycles in~\(T\) and the \(2\)-cocycles in~\(BT\)
both correspond to the \(\kk\)-module~\(\Hom(N,\kk)\) of additive functions~\(N\to\kk\).
Since there are no non-trivial coboundaries, this yields canonical isomorphisms
\begin{equation}
  \label{eq:iso-HomNkk-Z1-Z2}
\begin{split}
  \Hom(N,\kk) &= Z^{1}(T) = Z^{2}(BT) \\
  &= H^{1}(T) = H^{2}(BT).
\end{split}  
\end{equation}
Given an element~\(\alpha\) in~\(\Hom(N,\kk)\) or~\(H^{1}(T)\),
we write \(\hatalpha\in Z^{2}(BT)\) for the corresponding cocycle.

Let~\(\alpha_{1}\),~\(\alpha_{2}\colon N\to\Z\) be additive
and define functions~\(\alpha_{11}\),~\(\alpha_{12}\colon N\to\Z\) by
\begin{equation}
  \alpha_{11}(x) = \frac{\alpha_{1}(x)(\alpha_{1}(x)-1)}{2}
  \qquad\text{and}\qquad
  \alpha_{12}(x) = \alpha_{1}(x)\,\alpha_{2}(x)
\end{equation}
for~\(x\in N\).
Note that \(\alpha_{11}\) is well-defined
because one of the terms in the numerator is even.

\begin{lemma}
  \label{thm:d-q-ij}
  Let \(\hatalpha_{11}\),~\(\hatalpha_{12}\in C^{2}(BT)\)
  be the cochains corresponding to~\(\alpha_{11}\) and~\(\alpha_{12}\).
  \begin{enumroman}
  \item
    \label{thm:d-q-ij-1}
    The differentials of~\(\hatalpha_{11}\) and~\(\hatalpha_{12}\) satisfy
    \begin{equation*}
      d\hatalpha_{11} = \hatalpha_{1}\cupone\hatalpha_{1}
      \qquad\text{and}\qquad
      d\hatalpha_{12} = \hatalpha_{1}\cupone\hatalpha_{2}+\hatalpha_{2}\cupone\hatalpha_{1}.
    \end{equation*}
  \item
  \label{thm:d-q-ij-2}
    In \(\BB  C^{*}(T)\) one has
    \begin{equation*}
      d\,[\hatalpha_{11}] = [\hatalpha_{1}]\circ[\hatalpha_{1}]
      \qquad\text{and}\qquad
      d\,[\hatalpha_{12}] = [\hatalpha_{1}]\circ[\hatalpha_{2}]+[\hatalpha_{2}]\circ[\hatalpha_{1}].
    \end{equation*}
  \end{enumroman}
\end{lemma}

\begin{proof}
  As above, we write
  \(\sigma=\bigbarsim{\barsim{x,y},\barsim{z}}\) for a \(3\)-simplex in~\(BT\).
  Then
  \begin{equation}
    \partial_{0}\sigma = \barbarsim{z},
    \quad
    \partial_{1}\sigma = \barbarsim{y+z},
    \quad
    \partial_{2}\sigma = \barbarsim{x+y},
    \quad
    \partial_{3}\sigma = \barbarsim{x},
  \end{equation}
  by~\eqref{ed:d-3-simplex}, hence
  \begin{equation}
    \AWu{(1,2,1)}(\sigma) = \barbarsim{y+z}\otimes\barbarsim{x}-\barbarsim{x+y}\otimes\barbarsim{z},
  \end{equation}
  \cf~\citegersten{Ex.~\refgersten{ex:steenrod}}. 
  We therefore have
  \begin{equation}
    -(d\gamma)(\sigma) = \gamma(d\sigma)
    = \gamma\bigl(\barbarsim{z}-\barbarsim{y+z}+\barbarsim{x+y}-\barbarsim{x}\bigr)
  \end{equation}
  for any~\(\gamma\in C^{2}(BL)\) as well as
  \begin{equation}
  \begin{split}
    (\hatalpha_{1}\cupone\hatalpha_{2})(\sigma)
    &= -\bigl(\,\transpp{\AWu{(1,2,1)}}(\hatalpha_{1}\otimes\hatalpha_{2})\bigr)(\sigma)
    = -(\hatalpha_{1}\otimes\hatalpha_{2})\bigl(\AWu{(1,2,1)}(\sigma)\bigr) \\
    &= -\hatalpha_{1}(\barbarsim{y+z})\,\hatalpha_{2}(\barbarsim{x})
      - \hatalpha_{1}(\barbarsim{x+y})\,\hatalpha_{2}(\barbarsim{z}) \\
    &= -\alpha_{1}(y+z)\,\alpha_{2}(x)
      + \alpha_{1}(x+y)\,\alpha_{2}(z).
  \end{split}
  \end{equation}

  Using the definitions of~\(\alpha_{11}\) and~\(\alpha_{12}\), we get
  \begin{align}
    (d\hatalpha_{11})(\sigma)
    &= \alpha_{1}(y)\,\alpha_{1}(z-x)
    = (\hatalpha_{1}\cupone\hatalpha_{1})(\sigma), \\
  \begin{split}
    (d\hatalpha_{12})(\sigma)
    &= \alpha_{1}(y)\,\alpha_{2}(z-x) + \alpha_{1}(z-x)\,\alpha_{2}(y) \\
    &= (\hatalpha_{1}\cupone\hatalpha_{2})(\sigma) + (\hatalpha_{2}\cupone\hatalpha_{1})(\sigma),
  \end{split}    
  \end{align}
  which proves the first part.

  The second claim now follows from the definition~\(d\,[\gamma]=-[d\gamma]\)
  of the differential on a desuspension together with the identity~\eqref{eq:prod-bar-1-1}
  and the identification~\eqref{eq:def-cup-one}
  of the \(\cupone\)-product with the operation~\(\EE_{1}\).
\end{proof}

\begin{remark}
  \label{rem:cup-two}
  In terms of interval cut operations, the \(\cuptwo\)-product is given
  by the transpose of the signed surjection \(-(1,2,1,2)\), and
  \(\AWu{(1,2,1,2)}(\barbarsim{x})=-\barbarsim{x}\otimes\barbarsim{x}\)
  for any \(2\)-simplex~\(\barbarsim{x}\) in~\(BT\), \cf~\citegersten{Ex.~\refgersten{ex:steenrod}}.
  Hence \(\hatalpha_{12}=\hatalpha_{1}\cuptwo\hatalpha_{2}\),
  which explains the formula for~\(d\hatalpha_{12}\).
\end{remark}

\section{Partial quotients and toric varieties}
\label{sec:toric}

Let \(\Sigma\) be a simplicial complex or, more generally, a simplicial poset
on the vertex set~\(V\), \cf~\cite[Sec.~2.8]{BuchstaberPanov:2015}.
We always assume \(\Sigma\) and~\(V\) to be finite,
and we allow \newterm{ghost vertices}~\(v\in V\) not occurring in~\(\Sigma\).
We denote the complete simplicial complex on~\(V\) by the same letter.
The \newterm{folding map}~\(\Sigma\to V\) sends every~\(\sigma\in\Sigma\) to its vertex set.
Let \(\Sigma\) and~\(\Sigma'\) be two simplicial posets on the same vertex set~\(V\).
A \newterm{vertex-preserving morphism}~\(\Sigma\to\Sigma'\) is a morphism of posets
compatible with the folding maps.
The folding map itself is vertex-preserving,
as is the inclusion map of a subcomplex of a simplicial complex.

We follow our notation for simplicial sets and topological spaces
introduced in \Cref{sec:simp-general}.
The \newterm{moment-angle complex} of the simplicial poset~\(\Sigma\) is
the topological space
\begin{equation}
  \ZZ_{\Sigma} = \colim_{\sigma\in\Sigma}\ZZ_{\sigma}
\end{equation}
where
\begin{equation}
  \ZZ_{\sigma} = (\DD^{2})^{\sigma}\times(\SS^{1})^{V\setminus\sigma}.
\end{equation}
The compact torus~\(\TT=(\SS^{1})^{V}\) acts on~\(\ZZ_{\Sigma}\) in the canonical way.
Any vertex-preserving morphism~\(\Sigma\to\Sigma'\) of simplicial posets
canonically induces a \(\TT\)-equiv\-ari\-ant map~\(\ZZ_{\Sigma}\to\ZZ_{\Sigma'}\).
Let \(\KK\subset\TT\) be a closed subgroup with quotient~\(\LL\) of rank~\(n\)
and let \(\XX_{\Sigma}=\ZZ_{\Sigma}/\KK\) be the corresponding partial quotient.

For a simplicial rational fan~\(\Sigma\) in~\(\R^{n}\),
we freely use the same letter for the associated simplicial complex
on the vertex set~\(V\) which corresponds to the rays in~\(\Sigma\).
The canonical map~\(\R^{V}\to\R^{n}\) sends each ray~\(\rho\in\Sigma\) to its minimal integral representative~\(x_{\rho}\).
If \(\Sigma\) does not span \(\R^{n}\), we choose a basis for a lattice complement to~\(\lin\Sigma\,\cap\Z^{n}\)
and add corresponding ghost vertices to~\(V\). As a result,
the canonical map~\(\R^{V}\to\R^{n}\) is surjective.
We denote by~\(\hat\Sigma\) the subfan of the positive orthant in~\(\R^{V}\)
that is combinatorially equivalent to~\(\Sigma\) under this projection.

The toric variety~\(\XXX_{\Sigma}\) associated to the fan~\(\Sigma\) is an orbifold,
and its Cox construction~\(\ZZZ_{\Sigma}=\XXX_{\hat\Sigma}\subset\C^{V}\) is smooth.
We have~\(\XXX_{\Sigma}=\ZZZ_{\Sigma}/\KKK\)
for a possibly disconnected algebraic subgroup~\(\KKK\subset\TTT\)
acting with finite isotropy groups, \cf~\cite[Thm.~5.4.5]{BuchstaberPanov:2015}.
In this setting we set \(\LLL=\TTT/\KKK\), and we let \(\KK=\KKK\cap\TT\)
be the compact form of~\(\KKK\), so that \(\LL=\TT/\KK\).
As before, we define \(\XX_{\Sigma}=\ZZ_{\Sigma}/\KK\).

\begin{proposition}
  The map~\(\XX_{\Sigma}\hookrightarrow\XXX_{\Sigma}\)
  induced by the inclusion~\(\ZZ_{\Sigma}\hookrightarrow\ZZZ_{\Sigma}\)
  is an \(\LL\)-equivariant strong deformation retract.
\end{proposition}

In particular, \(\ZZ_{\Sigma}\hookrightarrow\ZZZ_{\Sigma}\)
is a \(\TT\)-equivariant strong deformation retract
\cite[Prop.~20]{Strickland:1999},~\cite[Thm.~4.7.5]{BuchstaberPanov:2015},
as mentioned already in the introduction.

\begin{proof}
  This is a consequence of the topological description
  of toric varieties given in~\cite{Franz:2010}.
  Recall that one can define a toric variety~\(\XX_{\Sigma}(k)\)
  over any submonoid~\(k\) of~\(\C\). We write \(\DDDD\subset\C\)
  for the unit disc, considered as a submonoid.
  Then \(\XXX_{\Sigma}=\XXX_{\Sigma}(\C)\) and
  \(\ZZ_{\Sigma}=\XXX_{\hat\Sigma}(\DDDD)\),
  see the remarks at the end of Section~4 in \cite{Franz:2010}.
  By~\cite[Thm.~2.1]{Franz:2010}, the inclusion~%
  \( 
    \XXX_{\Sigma}(\DDDD) \hookrightarrow \XXX_{\Sigma}(\C)
  \) 
  is an \(\LL\)-equivariant strong deformation retract.
  It therefore suffices to show that
  the canonical map~\(\ZZ_{\Sigma}=\XXX_{\hat\Sigma}(\DDDD)\to\XXX_{\Sigma}(\DDDD)\)
  is the quotient by~\(\KK\). We may assume that \(\Sigma\) has a single maximal cone~\(\sigma\).

  If \(\KK\) is finite, then it equals \(\KKK\), and the claim follows
  by naturality with respect to the inclusion~\(\DDDD\hookrightarrow\C\)
  since \(\XXX_{\sigma}=\XXX_{\hat\sigma}/\KKK\).
  If \(\KK\) is a freely acting subtorus of~\(\TT\), then
  \( 
    \ZZ_{\sigma} \cong \KK \times \XX_{\sigma}
  \), 
  and the claim is again true.
  In general, we may compute the quotient by~\(\KK\) in stages by first dividing out
  a finite subgroup~\(\Gamma\subset\KK\) and then a subtorus~\(\KK/\Gamma\subset\TT/\Gamma\)
  acting freely on~\(\ZZ_{\sigma}/\KK\).
  At each step, our claim holds, and so it does in general.
\end{proof}

As a consequence of this and
the topological variant of~\Cref{thm:HG-top-sing}\,\ref{thm:HG-top-sing-3},
we have natural isomorphisms 
\begin{equation}
  H^{*}(\XX_{\Sigma}) = H^{*}(\XXX_{\Sigma})
  \qquad\text{and}\qquad
  H^{*}_{\KK}(\ZZ_{\Sigma}) = H^{*}_{\KK}(\ZZZ_{\Sigma}) = H^{*}_{\KKK}(\ZZZ_{\Sigma}).
\end{equation}

If a torus acts on a topological space with finite isotropy groups whose orders
are invertible in~\(\kk\), then we call this action \newterm{\(\kk\)-free}
and the quotient~\(\XXX_{\Sigma}\) \newterm{\(\kk\)-smooth}.
The action of~\(\KKK\) on~\(\ZZZ_{\Sigma}\) is \(\kk\)-free
if and only if the \(\KK\)-action on~\(\ZZ_{\Sigma}\) is so.
This happens if and only if the fan~\(\Sigma\) is \(\kk\)-regular,
that is, if and only if the minimal integral representatives of the rays of any cone~\(\sigma\in\Sigma\)
can be completed to a basis for~\(H_{1}(\LLL)\cong\kk^{n}\).
It is also equivalent to~\(\XXX_{\Sigma}\) being a \(\kk\)-homology manifold.

\begin{lemma}
  \label{thm:free-quotient}
  If \(\KK\) acts \(\kk\)-freely on~\(\ZZ_{\Sigma}\),
  then there is an isomorphism of graded algebras
  \begin{equation*}
    H_{\KK}^{*}(\ZZ_{\Sigma}) = H^{*}(\XX_{\Sigma}),
  \end{equation*}
  natural with respect to vertex-preserving morphisms.
\end{lemma}

This is well-known. We give a short proof for the convenience of the reader.

\begin{proof}
  Let \(\Gamma\subset \KK\) be the subgroup generated by all isotropy groups of~\(\KK\)
  in~\(\ZZ=\ZZ_{\Sigma}\). It is finite and its order is in invertible in~\(\kk\).

  Given that \(\Gamma\) acts trivially on~\(H^{*}(\ZZ)\),
  the quotient~\(\ZZ\to \ZZ/\Gamma\) induces an isomorphism in cohomology,
  \cf~\cite[Thm.~III.7.2]{Bredon:1972}, as does the projection~\(B\KK\to B(\KK/\Gamma)\).
  As a consequence, we get an isomorphism in equivariant cohomology
  \begin{equation}
    H_{\KK/\Gamma}^{*}(\ZZ/\Gamma)\to H_{\KK}^{*}(\ZZ).
  \end{equation}

  By construction, the group~\(\KK/\Gamma\) acts freely on~\(\ZZ/\Gamma\),
  so that the canonical map
  \begin{equation}
    H^{*}(\ZZ/\KK)=H^{*}\bigl(\,(\ZZ/\Gamma)\bigm/(\KK/\Gamma)\,\bigr)
    \to H_{\KK/\Gamma}^{*}(\ZZ/\Gamma)
  \end{equation}
  is an isomorphism.
\end{proof}

By \Cref{thm:simplicial-torus-inclusion}, the projection~\(\TT\to\LL\)
corresponds to a map~\(T\to L\) between simplicial tori; let \(K\) be its kernel.
The induced map~\(K\to\KK\) is a homotopy equivalence
because again by \Cref{thm:simplicial-torus-inclusion}
it is so on the connected component of the identity.
Hence we conclude from \Cref{thm:HG-top-sing}\,\ref{thm:HG-top-sing-2} that
we have an isomorphism
\begin{equation}
  \label{eq:iso-HK-HKK}
  H_{K}^{*}(\ZZ_{\Sigma}) = H_{\KK}^{*}(\ZZ_{\Sigma})
\end{equation}
of algebras over~\(H^{*}(BK)=H^{*}(B\KK)\),
natural with respect to vertex-preserving morphisms of simplicial posets.

\section{Davis\texorpdfstring{--}{-}Januszkiewicz spaces and face rings}
\label{sec:DJ-facerings}

We continue to write \(\Sigma\) for a simplicial poset on the vertex set~\(V\).
Let \(\SimpSOne=B\Z\) be the simplicial circle
and \(T=(\SimpSOne)^{V}\) a simplicial torus.
For any~\(\sigma\in\Sigma\), we write~\(T^{\sigma}\subset T\) for the simplicial subtorus
corresponding to the vertices in~\(\sigma\). Let
\begin{equation}
  \SimpDJ_{\Sigma} = \colim_{\sigma\in\Sigma}\SimpDJ_{\sigma}
\end{equation}
be the simplicial Davis--Januszkiewicz space, where
\begin{equation}
  \SimpDJ_{\sigma} = BT^{\sigma}\times b_{0}^{V\setminus\sigma}\subset BT.
\end{equation}
The construction is natural with respect to vertex-preserving morphisms of simplicial posets.
In particular, the folding map of~\(\Sigma\) induces a map~\(\SimpDJ_{\Sigma}\to BT\).
Note that \(\SimpDJ_{\Sigma}\) is \(1\)-reduced since every~\(\SimpDJ_{\sigma}\) is so.

\begin{lemma}
  \label{thm:Borel-ZK-DJ}
  The simplicial Borel construction~\((\ZZ_{\Sigma})_{T}\) and the
  simplicial Davis--Januszkiewicz space~\(\SimpDJ_{\Sigma}\) are homotopy-equivalent,
  naturally with respect to vertex-preserving morphisms of simplicial posets.
  Moreover, the homotopy equivalences are compatible with the canonical maps to~\(BT\).
\end{lemma}

\begin{proof}
  The polyhedral product functor informally introduced in~\eqref{eq:intro-def-ma}
  can be extended to simplicial sets,
  \cf~\cite[Secs.~4.2~\&~8.1]{BuchstaberPanov:2015}. We have
  \begin{equation}
    \ZZ_{\Sigma} = \PolyProd{\Sigma}{\DD^{2}}{\SS^{1}}
    \qquad\text{and}\qquad
    \SimpDJ_{\Sigma} = \PolyProd{\Sigma}{B\SimpSOne}{b_{0}}.
  \end{equation}

  Since \(\DD^{2}\) is \(\SS^{1}\)-equivariantly contractible, the projection
  \begin{equation}
    (\DD^{2})_{\SimpSOne} \to B\SimpSOne
  \end{equation}
  is a homotopy equivalence,
  as is by \Cref{thm:HG-top-sing}\,\ref{thm:HG-top-sing-3} the composition
  \begin{equation}
    e_{0} \hookrightarrow E\SimpSOne = (\SimpSOne)_{\SimpSOne} \hookrightarrow (\SS^{1})_{\SimpSOne}.
  \end{equation}
  The simplicial analogue of~\cite[Prop.~4.2.3]{BuchstaberPanov:2015}
  now shows that both arrows in the zigzag
  \begin{equation}
    (\ZZ_{\Sigma})_{T}
    = \bigPolyProd{\Sigma}{(\DD^{2})_{\SimpSOne}}{(\SS^{1})_{\SimpSOne}}
    \leftarrow \bigPolyProd{\Sigma}{(\DD^{2})_{\SimpSOne}}{e_{0}}
    \to \PolyProd{\Sigma}{BS^1}{b_{0}} = \SimpDJ_{\Sigma}
  \end{equation}
  are homotopy equivalences.
  The naturality of the zigzag follows from that of the polyhedral product functor.

  For the compatibility with the maps to~\(BT\) we use the naturality with respect to the folding map
  and the commutative diagram
  \begin{equation}
    \begin{tikzcd}
      \bigl((\DD^{2})^{V}\bigr)_{T} \arrow{d} & \bigl((\DD^{2})^{V}\bigr)_{T} \arrow[equal]{l} \arrow{r} \arrow{d} & BT \arrow[equal]{d} \\
      BT & BT \arrow[equal]{l} \arrow[equal]{r} & BT \mathrlap{.}
    \end{tikzcd}
    \qedhere
  \end{equation}
\end{proof}

We can now state our first dga model for~\((\ZZ_{\Sigma})_{\KK}\),
using the one-sided bar construction defined in \Cref{sec:bar}.

\begin{proposition}
  \label{thm:quiso-bar-DJ}
  The dgas~\(C^{*}((\ZZ_{\Sigma})_{\KK})\)
  and~\(\Bar{C^{*}(\SimpDJ_{\Sigma})}{C^{*}(BL)}\)
  are quasi-iso\-mor\-phic, naturally with respect to vertex-preserving morphisms.
\end{proposition}

\begin{proof}
  We know from~\eqref{eq:iso-HK-HKK}
  that \((\ZZ_{\Sigma})_{\KK}\) is quasi-iso\-mor\-phic to~\((\ZZ_{\Sigma})_{K}\).
  By \Cref{thm:dga-bar}, the dgas~\(C^{*}((\ZZ_{\Sigma})_{K})\)
  and~\(\Bar{C^{*}((\ZZ_{\Sigma})_{T})}{C^{*}(BL)}\)
  are quasi-iso\-mor\-phic.
  Using \Cref{thm:Borel-ZK-DJ},
  we can pass from~\((\ZZ_{\Sigma})_{T}\)
  to~\(\SimpDJ_{\Sigma}\).
  By \Cref{thm:Psi-quiso},
  we get a quasi-iso\-mor\-phism at each step.
\end{proof}

We refer to~\cite[Def.~3.5.2]{BuchstaberPanov:2015} for the definition
of the face ring~\(\kk[\Sigma]\) of the simplicial poset~\(\Sigma\).
This construction generalizes the Stanley-Reisner ring of a simplicial complex.
Its functoriality with respect to vertex-preserving morphisms
induces a morphism of algebras~\(\kk[V]\to\kk[\Sigma]\).
We use an even grading of face rings, so that
the polynomial ring~\(\kk[V]=\kk[\,t_{v}\,|\,v\in V\,]\) is generated in degree~\(2\).

The cohomology algebra~\(H^{*}(DJ_{\Sigma})\) is isomorphic to~\(\kk[\Sigma]\),
see \cite[Prop.~4.3.1, Ex.~4.10.8]{BuchstaberPanov:2015} for the topological case
and \citegersten{Sec.~\refgersten{sec:dj}} for the simplicial setting.
This isomorphism is natural with respect to vertex-preserving morphisms.
In particular, the map~\(H^{*}(BT)\to H^{*}(DJ_{\Sigma})\)
corresponds to the morphism~\(\kk[V]\to\kk[\Sigma]\).

The following
result from~\cite{Franz:2018a} is crucial for us, see~\citegersten{Thm.~\refgersten{thm:formality-DJ-simp}}.
Recall that \(\kk[V]\) and~\(\kk[\Sigma]\) are canonically hgas
with trivial differentials and trivial operations~\(\EE_{k}\) for~\(k\ge1\).

\begin{theorem}
  \label{thm:formality-DJ-short}
  There is a quasi-iso\-mor\-phisms of hgas~\(f_{\Sigma}\colon C^{*}(DJ_{\Sigma})\to\kk[\Sigma]\),
  natural with respect to vertex-preserving morphisms of simplicial posets.
\end{theorem}

The naturality of~\(f_{\Sigma}\) with respect to the folding map~\(DJ_{\Sigma}\to BT\)
implies that the diagram
\begin{equation}
  \begin{tikzcd}
    C^{*}(DJ_{\Sigma}) \arrow{d}[left]{f_{\Sigma}} & C^{*}(BT) \arrow{d}{f_{T}} \arrow{l} \\
    \kk[\Sigma] & \kkV  \arrow{l}
  \end{tikzcd}
\end{equation}
commutes.

Let \((y_{v})\) be the canonical basis for \(H_{1}(T;\Z)=\Z^{V}\)
and let \((x_{v})\) be their images in~\(H_{1}(L;\Z)\).
The maps~\(f_{T}\) and~\(f_{\Sigma}\) both depend on the choice
of representatives~\( c_{v}\in C_{1}(BT)\) for the basis elements~\(y_{v}\),
see the construction in~\citegersten{Sec.~\refgersten{sec:formality-BT}}.
Any choice of representative~\( c\in C_{1}(\SimpSOne;\Z)\) for~\(1\in H_{1}(\SimpSOne;\Z)=\Z\)
canonically leads to representatives by embedding \( c\) into the various circle factors of~\(T\).
In what follows, we use the tautological representative~\( c=[1]\in C_{1}(\SimpSOne;\Z)\)
described in \Cref{sec:tori} unless stated otherwise.

Combining \Cref{thm:quiso-bar-DJ} with \Cref{thm:formality-DJ-short}
and \Cref{thm:Psi-quiso},
we obtain our second dga model for~\((\ZZ_{\Sigma})_{K}\).

\begin{proposition}
  \label{thm:HKZZSigma-Tor}
  The map~\(f_{\Sigma}\) induces a quasi-iso\-mor\-phism of dgas
  \begin{equation}
    \PsiSigma \colon \Bar{C^{*}(DJ_{\Sigma})}{C^{*}(BL)}
    \to \BB_{f_{\Sigma}}(\kk,C^{*}(BL),\kk[\Sigma]),
  \end{equation}
  natural with respect to vertex-preserving morphisms of simplicial posets.
  In particular, the dga~\(\BB_{f_{\Sigma}}(\kk,C^{*}(BL),\kk[\Sigma])\) is quasi-iso\-mor\-phic to~\(C^{*}((\ZZ_{\Sigma})_{K})\).
\end{proposition}

Note that the product on~\(\BB_{f_{\Sigma}}(\kk,C^{*}(BL),\kk[\Sigma])\) is the componentwise product
\begin{equation}
  [a_{1}|\dots|a_{k}]f \circ [b_{1}|\dots|b_{l}]g =
  \bigl([a_{1}|\dots|a_{k}]\circ [b_{1}|\dots|b_{l}]\bigr)fg
\end{equation}
where \([a_{1}|\dots|a_{k}]\),~\([b_{1}|\dots|b_{l}]\in\BB\,C^{*}(BL)\) and \(f\),~\(g\in\kk[\Sigma]\).
The map~\(\PsiSigma\) is multiplicative
because the terms involving the operations~\(\EE_{k}\) with~\(k\ge1\)
in the formula given in \Cref{thm:def-prod-bar} are annihilated by the hga map~\(C^{*}(DJ_{\Sigma})\to\kk[\Sigma]\).

In Sections~\ref{sec:twisted-prod} and~\ref{sec:2-inv}
we will need to know how to evaluate the quasi-iso\-mor\-phism \(f_{T}\)
on \(2\)-cochains coming from~\(BL\).

\begin{lemma}
  \label{thm:value-fT}
  Let \(\gamma\in C^{2}(BL)\), and let \(\alpha\colon H_{1}(L;\Z)\to\kk\)
  be the function corresponding to it under the isomorphism~\eqref{eq:iso-HomNkk-Z1-Z2}.
  \begin{enumroman}
  \item
    \label{thm:value-fT-1}
    Assume that \(c=[1]\in C_{1}(\SimpSOne;\Z)\) is the tautological representative.
    Then the value of~\(f_{T}\) on the pull-back of~\(\gamma\) to~\(BT\)
    is given by
    \begin{equation*}
      \sum_{v\in V}\alpha(x_{v})\,t_{v}.
    \end{equation*}
  \item
    \label{thm:value-fT-2}
    If \( c=\sum  c_{m}[m]\in C_{1}(\SimpSOne)\) is an arbitrary representative,
    then for any~\(v\in V\) the coefficient in front of~\(t_{v}\) becomes
    \begin{equation*}
      \sum c_{m}\,\alpha(m x_{v}).
    \end{equation*}
  \end{enumroman}
\end{lemma}

\begin{proof}
  Let \(\beta\in C^{2}(BT)\) be the pull-back of~\(\gamma\),
  and let \(\tilde\beta\colon H_{1}(T;\Z)\to\kk\) be the corresponding function.

  By inspection of the construction of~\(f_{T}\) in~\citegersten{Sec.~\refgersten{sec:construction-f}}
  (where it is called \(\ffbar^{*})\), the coefficient of~\(t_{v}\) in~\(f_{T}(\beta)\)
  is the value of~\(\beta\) on the \(2\)-chain
  \begin{equation}
    b_{v} = \pi_{*} S( c_{v}\cdot e_{0})
  \end{equation}
  where \(\pi\colon ET\to BT\) is the universal \(T\)-bundle,
  \(S\colon C(ET)\to C(ET)\) the canonical homotopy contracting
  \(ET\) to the basepoint~\(e_{0}\). Moreover,
  \( c_{v}\cdot e_{0}\) refers to the action of~\(C(G)\) on~\(C(EG)\).

  If \(c=[1]\), we have by the definition of all the objects just mentioned that
  \begin{align}
    b_{v}
    &= \pi_{*} S( c_{v}\cdot e_{0})
    = \pi_{*} S( c_{v}\cdot \barsim{1_{0}})
    = \pi_{*} S(\barsim{\barsim{y_{v}},1_{0}}) \\
    &= \pi_{*}(\barsim{1_{2},\barsim{y_{v}},1_{0}})
    = \barsim{\barsim{y_{v}},1_{0}} = \barbarsim{y_{v}},
  \end{align}
  \cf~\citegersten{Sec.~\refgersten{sec:universal-bundles}} and~\cite[\S 21]{May:1968}.
  In other words, \(b_{v}\) is the \(2\)-simplex corresponding to~\(c_{v}\)
  via the bijections~\eqref{eq:iso-N-T} and~\eqref{eq:iso-N-BT}.
  This implies \(\beta(b_{v})=\tilde\beta(y_{v})\) by the isomorphisms~\eqref{eq:iso-HomNkk-Z1-Z2}.
  Since these isomorphisms are natural in~\(T\), we get \(\beta(b_{v})=\alpha(x_{v})\).
  This proves the first part.

  If \(c\) is a linear combination of simplices, then
  we have to do the above computation for any simplex~\([m]\) appearing in~\(c\)
  and form the resulting linear combination.
  Under the bijection~\eqref{eq:iso-N-T}, the simplex~\([m x_{v}]\) corresponds
  to~\(m x_{v}\), which gives the second claim.
\end{proof}

\section{The Koszul complex}
\label{sec:koszul}

We define the dga~\(\Kl_{\Sigma}\) as the Koszul complex of the \(H^{*}(BL)\)-algebra~\(\kk[\Sigma]\).
This is the tensor product
\begin{equation}
  \Kl_{\Sigma} = H^{*}(L)\otimes \kk[\Sigma]
\end{equation}
with componentwise multiplication.
We write elements of~\(\Kl_{\Sigma}\) in the form~\(\alpha f\) with~\(f\in\kk[\Sigma]\)
and~\(\alpha\in H^{*}(L)\).
The differential is determined by the composition \(H^{*}(BL)\to \kkV \to\kk[\Sigma]\);
one has
\begin{equation}
  \label{eq:def-d-Kl}
  d f = 0,
  \qquad\text{and}\quad
  d\alpha = -\sum_{v\in V}\alpha( x_{v})\,t_{v}
  \quad\text{if \(\deg{\alpha}=1\).}
\end{equation}
Recall that \( x_{v}\in H_{1}(L;\Z)\) is the image of
the \(v\)-th canonical generator of~\(H_{1}(T;\Z)\) under the projection~\(T\to L\) with kernel~\(K\).
The assignment~\(\Sigma\mapsto\Kl_{\Sigma}\) is contravariant with respect to
vertex-preserving morphisms of simplicial posets because the latter are compatible with folding maps
and therefore induce morphisms of \(\kk[V]\)-algebras between face rings.

\begin{proposition}
  \label{thm:dga-model-ZZ}
  Assume that \(K=1\), so that \(T=L\).
  Then \(\Kl_{\Sigma}\) is quasi-iso\-mor\-phic to~\(C^{*}(\ZZ_{\Sigma})\) as a dga.
\end{proposition}

This answers a question posed by Berglund~\cite[Question~5]{Berglund:2010}.
This result (with a shorter proof) is already implicit in~\cite[Thm.~1.3]{Franz:2003a}.
Using rational coefficients and polynomial differential forms,
Panov--Ray~\cite[Sec.~6]{PanovRay:2008} showed that \(\Kl_{\Sigma}\)
is quasi-iso\-mor\-phic to~\(\APL^{*}(\ZZ_{\Sigma})\) as a cdga;
we will come back to this in \Cref{sec:realcoeffs}.

\begin{proof}
  Analogously to \Cref{thm:HKZZSigma-Tor},
  we get from our assumptions, \Cref{thm:formality-DJ-short} and
  the obvious generalization of \Cref{thm:Psi-quiso} a quasi-iso\-mor\-phism of dgas
  \begin{equation}
    \Bar{C^{*}(DJ_{\Sigma})}{C^{*}(BT)} \to \BB_{f_{\Sigma}}(\kk,\kkV,\kk[\Sigma]).
  \end{equation}
  The map
  \begin{equation}
    \label{eq:map-HT-BkV}
    H^{*}(T) \to \BB (\kkV),
    \quad
    \alpha_{v_{1}}\cdots\alpha_{v_{k}} \mapsto
    [t_{v_{k}}]\circ\dots\circ[t_{v_{k}}]
  \end{equation}
  is known to be a quasi-iso\-mor\-phism of dgas.
  (Here we assume that our basis elements for~\(H^{1}(T)\)
  correspond to those for~\(H^{2}(BT)\).
  The map~\eqref{eq:map-HT-BkV}, however, is independent of the chosen bases.)
  Note that the product on~\(\BB (\kkV )\)
  is the usual shuffle product as the hga structure is trivial.

  A standard spectral sequence argument extends this to the map of dgas
  \begin{equation}
    \label{eq:map-Kl-barHBT}
  \begin{split}
    \Kl_{\Sigma}=H^{*}(T)\otimes\kk[\Sigma] &\to \BB_{f_{\Sigma}}(\kk,\kkV,\kk[\Sigma]),
    \\
    \alpha_{v_{1}}\cdots\alpha_{v_{k}} f &\mapsto
    [t_{v_{k}}]\circ\dots\circ[t_{v_{k}}]\circ f.
  \end{split}    
  \end{equation}
  Note that it is a chain map by the definition of the differentials~\eqref{eq:def-d-Kl}
  and~\eqref{eq:def-d-bar}.
  We conclude with \Cref{thm:quiso-bar-DJ}.
\end{proof}

The quasi-iso\-mor\-phisms in \Cref{thm:formality-DJ-short} depend
on the choice of a representative~\( c\in C_{1}(\SimpSOne;\Z)\)
of the canonical generator of~\(H_{1}(\SimpSOne;\Z)=\Z\).
What stops us from extending \Cref{thm:dga-model-ZZ}
to an arbitrary quotient~\(T\to L\) is that
it might be (and, an view of \Cref{ex:intro}, generally is) impossible
to choose representatives in~\(C_{1}(L)\) of a basis for~\(H_{1}(L;\Z)\)
such that the resulting quasi-iso\-mor\-phism~\(f_{L}\colon C^{*}(BL)\to H^{*}(BL)\)
makes the left square in the diagram
\begin{equation}
  \begin{tikzcd}
    C^{*}(BL) \arrow{d}{f_{L}} \arrow{r} & C^{*}(BT) \arrow{d}{f_{T}} \arrow{r} & C^{*}(DJ_{\Sigma}) \arrow{d}{f_{\Sigma}} \\
    H^{*}(BL) \arrow{r} & \kkV \arrow{r} & \kk[\Sigma]
  \end{tikzcd}
\end{equation}
commute.
Given that we are interested in the multiplicative structure,
we use a different approach based on \Cref{thm:HKZZSigma-Tor}.

Choose a basis~\( x_{1}\),~\dots,~\( x_{n}\) for~\(H_{1}(L;\Z)\)
and let \(\alpha_{1}\),~\dots,~\(\alpha_{n}\in H^{1}(L;\Z)\)
be the dual basis.
We write \(\coord{i}{x}=\alpha_{i}(x)\) for the corresponding
\(i\)-th coordinate of~\(x\in H_{1}(L)\).
Recall that \(x_{v}\in H_{1}(L;\Z)\) with~\(v\in V\)
is the image of the \(v\)-th canonical basis vector of~\(H_{1}(T;\Z)\).
An ordering of the \(m=|V|\)~vertices in~\(V\) gives the \newterm{characteristic matrix}
\begin{equation}
  \Lambda =
  \begin{bmatrix}
    x_{v_{1}}^{1} & \dots & x_{v_{m}}^{1} \\
    \vdots & \ddots & \vdots \\
    x_{v_{1}}^{n} & \dots & x_{v_{m}}^{n}
  \end{bmatrix}
  \in \Z^{n\times m}
\end{equation}
encoding the projection map~\(T\to L\).

By~\eqref{eq:iso-HomNkk-Z1-Z2}, our basis for~\(H^{1}(L;\Z)\) corresponds
to a basis for~\(H^{2}(BL;\Z)\)
with canonical representatives~\(\gamma_{1}=\hatalpha_{1}\),~\dots,~\(\gamma_{n}=\hatalpha_{n}\in C^{2}(BL;\Z)\).
We define the map
\begin{equation}
  \label{eq:def-Phi}
\begin{split}
  \Phi_{\Sigma}\colon \Kl_{\Sigma} = H^{*}(L)\otimes\kk[\Sigma] &\to \BB_{f_{\Sigma}}(\kk,C^{*}(BL),\kk[\Sigma]),
  \\
  \alpha_{i_{1}}\cdots\alpha_{i_{k}} f &\mapsto
  [\gamma_{i_{1}}]\circ\dots\circ[\gamma_{i_{k}}]\circ f
\end{split}  
\end{equation}
where~\(1\le i_{1}<\dots<i_{k}\le n\).

\begin{proposition}
  \label{thm:iso-tor-bar}
  The map~\(\Phi_{\Sigma}\) is a quasi-iso\-mor\-phism of complexes, natural with respect
  to vertex-preserving morphisms of simplicial posets.
  In particular, it induces an isomorphism of graded \(\kk\)-modules
  \begin{equation*}
    \HK^{*}(\ZZ_{\Sigma}) \cong \Tor_{H^{*}(BL)}(\kk[\Sigma],\kk).
  \end{equation*}
\end{proposition}

For simplicial complexes,
this was proven in~\cite[Thm.~3.3.2]{Franz:2001},
see also~\cite[Thm.~1.2]{Franz:2003a} and~\cite[Thm.~1.2]{Franz:2006}.
Note that \(\Phi_{\Sigma}\) does not preserve products for~\(L\ne1\)
since \(C^{*}(BL)\) is not (graded) commutative.

\begin{proof}
  Even if \(\Phi_{\Sigma}\) is not multiplicative,
  we can imitate the proof of \Cref{thm:dga-model-ZZ}.
  
  Both~\(\Kl_{\Sigma}\) and the one-sided bar construction being dgas
  with~\(\kk[\Sigma]\) contained in their centres,
  the compatibility of~\(\Phi_{\Sigma}\) with differentials follows as before
  from the definitions~\eqref{eq:def-d-BAA} and~\eqref{eq:def-d-Kl}
  because \(d[\gamma_{i}]\) is the image of the cocycle~\(\gamma_{i}\) in~\(\kk[\Sigma]\).

  Passing to the usual spectral sequences, we get on the first page
  the generalization
  \begin{equation}
    \Kl_{\Sigma} = H^{*}(L)\otimes\kk[\Sigma] \to \Bar{\kk[\Sigma]}{H^{*}(BL)}
  \end{equation}
  of the map~\eqref{eq:map-Kl-barHBT} to the case of arbitrary~\(L\),
  which becomes an isomorphism on the second page.
\end{proof}

\begin{remark}
  \label{rem:finitedim-complex}
  If the \(K\)-action on~\(\ZZ_{\Sigma}\) is \(\kk\)-free, 
  then there is a finite-dimensional subcomplex~\(\MM_{\Sigma}\subset\Kl_{\Sigma}\)
  such that the inclusion is a quasi-iso\-mor\-phism.
  This was observed in~\cite[Lemma~6.1]{Franz:2003a},~\cite[Lemma~6.1]{Franz:2006}.
  
  For moment-angle complexes, this subcomplex allows
  to prove Hochster's decomposition
  of the torsion product~\(\Tor_{H^{*}(BT)}(\kk,\kk[\Sigma])\)
  into the reduced cohomology of full subcomplexes of~\(\Sigma\),
  see~\cite[Thm.~3.2.4]{BuchstaberPanov:2015}.
  One can also turn~\(\MM_{\Sigma}\) into a dga
  such that the canonical projection~\(\Kl_{\Sigma}\to\MM_{\Sigma}\) is multiplicative,
  which leads to Baskakov's product formula, \cf~\cite[Thm.~4.5.7]{BuchstaberPanov:2015}.
  It is unclear how to generalize these results to partial quotients
  because there is no canonical projection anymore.
\end{remark}

\section{A twisted product on the Koszul complex}
\label{sec:twisted-prod}

We have remarked already that the deficiency of the map~\(\Phi_{\Sigma}\)
defined in \Cref{thm:iso-tor-bar} is that it is not multiplicative.
We now address this problem
in a way that is motivated by the discussion in \Cref{sec:tori}, in particular \Cref{thm:d-q-ij},
as well as by the identities~\eqref{eq:d-gamma-ii-h} and~\eqref{eq:d-gamma-ij-h} in the proof of \Cref{thm:iso-tor-general} below.

We consider the collection of functions~\(\alpha_{ij}\colon H_{1}(L;\Z)\to\Z\)
with~\(1\le j\le i\le n\) whose values on~\(x\in H_{1}(L;\Z)\) are given by
\begin{equation}
  \label{eq:def-gamma-ij}
  \alpha_{ij}(x) =
  \begin{cases}
    \frac{1}{2}\coord{i}{x}(\coord{i}{x}-1) & \text{for~\(i=j\),} \\
    \coord{i}{x}\coord{j}{x} & \text{for~\(i>j\).}
  \end{cases}
\end{equation}
Recall that \(x^{i}=\alpha_{i}(x)\) is
the \(i\)-th coordinate of~\(x\) with respect to the chosen basis for~\(H_{1}(L;\Z)\).
Because the dual basis elements~\(\alpha_{i}\) correspond to the cocycles~\(\gamma_{i}=\hat\alpha_{i}\in C^{2}(BL)\)
under the isomorphism~\eqref{eq:iso-HomNkk-Z1-Z2}, \Cref{thm:d-q-ij}\,\ref{thm:d-q-ij-2} tells us that
the cochains~\(\gamma_{ij}=\hatalpha_{ij}\in C^{2}(BL)\) satisfy
\begin{equation}
  \label{eq:d-gamma-ij}
  d\,[\gamma_{ii}] = [\gamma_{i}]\circ[\gamma_{i}]
  \qquad\text{and}\qquad
  d\,[\gamma_{ij}] = [\gamma_{i}]\circ[\gamma_{j}]+[\gamma_{j}]\circ[\gamma_{i}]
\end{equation}
for~\(i>j\), and by \Cref{thm:value-fT}\,\ref{thm:value-fT-1}
their images in~\(\kkZ[V]\) are
\begin{equation}
  \label{eq:twisting-general}
  q_{ij} =
  \begin{cases}
  \frac{1}{2}\sum_{v\in V}\coord{i}{ x_{v}}(\coord{i}{ x_{v}}-1)\,t_{v} & \text{if~\(i=j\),} \\
  \sum_{v\in V}\coord{i}{ x_{v}}\coord{j}{ x_{v}}\,t_{v} & \text{if~\(i>j\).}
  \end{cases}
\end{equation}

Using these elements, we define the map
\begin{equation}
  \label{eq:def-twisted-product}
  \tilde\mu = \mu_{\Kl}\prod_{i\ge j}\Bigl(\id-q_{ij}\,\iota(x_{i})\otimes\iota(x_{j})\Bigr)
  \colon \Kl_{\Sigma}\otimes\Kl_{\Sigma} \to \Kl_{\Sigma}
\end{equation}
where \(\mu_{\Kl}\) is the canonical product on~\(\Kl_{\Sigma}\) and
\begin{equation}
  \iota(x)\colon \Kl_{\Sigma}\to\Kl_{\Sigma},
  \quad
  \alpha f \mapsto (x\cdot\alpha) f
\end{equation}
the contraction mapping.
Because two operators~\(\iota(x)\) and~\(\iota(y)\) with \(x\),~\(y\in H_{1}(L)\) anticommute,
all operators~\(\iota(x_{i})\otimes\iota(x_{j})\) commute
in~\(\Kl_{\Sigma}\otimes\Kl_{\Sigma}\) by the Koszul sign rule,
so that the order in which they are applied does not matter.
For instance, for any~\(i\) and~\(j\) we have
\begin{equation}
  \label{eq:example-mult}
  \alpha_{i}*\alpha_{j} =
  \begin{cases}
    \alpha_{i}\alpha_{j} & \text{if \(i<j\),} \\
    q_{ii} & \text{if \(i=j\),} \\
    -\alpha_{j}\alpha_{i}+q_{ij} & \text{if \(i>j\).} \\
  \end{cases}
\end{equation}
We defer further examples for a moment
and show first that the multiplication~\eqref{eq:def-twisted-product} is associative.

\begin{lemma}
  \label{thm:twisted-dga}
  The map~\(\tilde\mu\) defines an associative product~\(*\),
  thus giving \(\Kl_{\Sigma}\) another dga structure besides the canonical one.
  The new product is again natural with respect to vertex-preserving morphisms of simplicial posets.
\end{lemma}

This actually holds for any choice of twisting elements~\(q_{ij}\in\kk[V]\) (of degree~\(2\)).

\begin{proof}
  Each operator~\(\iota(x_{i})\) anticommutes with the differential,
  which implies that \(\iota(x_{i})\otimes\iota(x_{j})\) is a chain map
  and therefore also \(\tilde\mu\) by~\(\kk[\Sigma]\)-bilinearity.
  The element~\(1\in\Kl_{\Sigma}\) is a unit for the new product
  because \(\iota(x_{i})\,1=0\) for any~\(i\).
  
  To prove associativity, we expand
  \begin{equation}
  \begin{split}
    \mathrlap{\tilde\mu\,(\id\otimes\tilde\mu)}\quad \\
    &= \mu_{\Kl}\prod_{i\ge j}\Bigl(\id-q_{ij}\,\iota(x_{i})\otimes\iota(x_{j})\Bigr)
    (\id\otimes\mu_{\Kl})\prod_{k\ge l}\Bigl(\id-q_{kl}\,\id\otimes\iota(x_{k})\otimes\iota(x_{l})\Bigr) \\
    &= \iter{\mu_{\Kl}}{3}\prod_{i\ge j}\Bigl(\id-q_{ij}\,\iota(x_{i})\otimes\iota(x_{j})\otimes\id
    -q_{ij}\,\iota(x_{i})\otimes\id\otimes\iota(x_{j})\Bigr) \\
    &\qquad\qquad \prod_{k\ge l}\Bigl(\id-q_{kl}\,\id\otimes\iota(x_{k})\otimes\iota(x_{l})\Bigr)
  \end{split}    
  \end{equation}
  where \(\iter{\mu_{\Kl}}{3}\) denotes the triple product.
  Here we have used that each contraction mapping is a derivation of the
  canonical product. Now
  \begin{multline}
    \id-q_{ij}\,\iota(x_{i})\otimes\iota(x_{j})\otimes\id
    -q_{ij}\,\iota(x_{i})\otimes\id\otimes\iota(x_{j}) \\
    = \bigl(\id-q_{ij}\,\iota(x_{i})\otimes\iota(x_{j})\otimes\id\bigr)
    \bigl(\id-q_{ij}\,\iota(x_{i})\otimes\id\otimes\iota(x_{j})\bigr)
  \end{multline}
  since \(\iota(x_{i})\iota(x_{i})=0\). Hence \(\tilde\mu\,(\id\otimes\tilde\mu)\)
  equals \(\mu_{\Kl}^{[3]}\) composed with all operators of the form
  \begin{equation}
  \begin{split}
    \id-q_{ij}\,\iota(x_{i})\otimes\iota(x_{j})\otimes\id , \\
    \id-q_{ij}\,\iota(x_{i})\otimes\id\otimes\iota(x_{j}) , \\
    \id-q_{ij}\id\otimes\iota(x_{i})\otimes\iota(x_{j}) \;\,
  \end{split}    
  \end{equation}
  with~\(i\ge j\) (in any order). Computing \(\tilde\mu\,(\tilde\mu\otimes\id)\)
  in the same way leads to the same result, showing that \(\tilde\mu\) is associative.

  The new product is natural because the twisting terms~\(q_{ij}\) live in~\(\kkZ[V]\),
  and vertex-preserving morphisms of simplicial posets are compatible with the folding maps.
\end{proof}

Explicitly,
a product~\(\alpha f*\beta g\)
with~\(\alpha\in H^{1}(L)\) can be written as
\begin{equation}
  \label{eq:twisted-product-deg-1}
  \alpha f*\beta g = \alpha\beta fg
  + \sum_{i\ge j} (x_{i}\cdot\alpha)(x_{j}\cdot\beta)\,q_{ij} fg \,;
\end{equation}
the general case~\(\alpha=\alpha_{i_{1}}\!\cdots\alpha_{i_{k}}\)
with~\(i_{1}<\dots<i_{k}\) can be reduced to this by associativity
since \(\alpha=\alpha_{i_{1}}*\,\cdots\,*\alpha_{i_{k}}\).
For instance,
\begin{equation}
\begin{split}
  \alpha_{1}\alpha_{3} * \alpha_{1}\alpha_{2}
  &= \alpha_{1}*\bigl(\alpha_{3} * \alpha_{1}\alpha_{2}\bigr) \\
  &= \alpha_{1}*\bigl(\alpha_{1}\alpha_{2}\alpha_{3}+q_{31}\,\alpha_{2}-q_{32}\,\alpha_{1}) \\
  &= q_{11}\,\alpha_{2}\alpha_{3} + q_{31}\,\alpha_{1}\alpha_{2} - q_{11}q_{32}.
\end{split}  
\end{equation}

\begin{theorem}
  \label{thm:iso-tor-general}
  The map~\(H^{*}(\Phi_{\Sigma})\) induces an isomorphism of graded algebras
  \begin{equation*}
    H_{\KK}^{*}(\ZZ_{\Sigma}) \cong \Tor_{H^{*}(BL)}(\kk,\kk[\Sigma])
  \end{equation*}
  where the \(\Tor\)~term carries the product induced by the \(*\)-product on~\(\Kl_{\Sigma}\).
\end{theorem}

\begin{proof}
  We show that the quasi-iso\-mor\-phism~\(\Phi=\Phi_{\Sigma}\)
  is multiplicative up to homotopy with respect to the twisted product on~\(\Kl_{\Sigma}\).

  Before we construct the \(\kk[\Sigma]\)-bilinear homotopy
  \begin{equation}
    H\colon \Kl_{\Sigma}\otimes\Kl_{\Sigma} \to \BB_{f_{\Sigma}}(\kk,C^{*}(BL),\kk[\Sigma]),
  \end{equation}
  we define an \(H^{*}(BK)\)-bilinear map~\(H_{1}\) with the same domain
  and codomain as follows: Let \(\alpha\),~\(\beta\in H^{*}(L)\). Write
  \(\beta=\alpha_{j_{1}}\cdots\alpha_{j_{l}}\) with~\(1\le j_{1}<\dots<j_{l}\le n\)
  and assume that \(\alpha\) is of the form~\(\tilde\alpha\,\alpha_{i}\)
  for some~\(1\le i\le n\) and
  some~\(\tilde\alpha\in\bigwedge(\alpha_{1},\dots,\alpha_{i-1})\).
  In this case we set
  \begin{equation}
    H_{1}(\alpha\otimes\beta) = (-1)^{\deg{\tilde\alpha}}
    \sum_{r=1}^{s}
    \Phi(\tilde\alpha)\circ\Phi(\alpha_{j_{1}}\cdots\alpha_{j_{r-1}})
    \circ[\gamma_{i j_{r}}]\circ\Phi(\alpha_{j_{r+1}}\cdots\alpha_{j_{l}})
  \end{equation}
  where \(s\) is the largest index such that \(j_{s}\le i\),
  or equal to~\(0\) if no such~\(s\) exists.
  The purpose of~\(H_{1}\) is to move \([\gamma_{i}]\) past all
  smaller factors~\([\gamma_{j_{r}}]\) with respect to our ordering of the basis.

  Now \(H\) is recursively defined by
  \begin{align}
    H(1\otimes \beta) &= 0, \\
    H(\alpha \otimes \beta) &=
      H_{1}(\alpha \otimes \beta) + H(\tilde\alpha\otimes\alpha_{i} * \beta)
  \end{align}
  for~\(\alpha\) and~\(\beta\) as above.
  Using the identities~\eqref{eq:d-gamma-ij}
  as well as the definition~\eqref{eq:def-twisted-product} of the twisted product,
  it is straightforward, albeit somewhat lengthy, to verify that
  \begin{equation}
    (dH+Hd)(\alpha\otimes\beta) = \Phi(\alpha)\circ\Phi(\beta)-\Phi(\alpha*\beta),
  \end{equation}
  as desired.
  The key point is that by the definition of~\(\gamma_{ij}\) and~\(q_{ij}\)
  as well as by that of the dga structure on the one-sided bar construction
  and~\eqref{eq:prod-bar-1-1} we have
  \begin{equation}
    \label{eq:d-gamma-ii-h}
    d[\gamma_{ii}]
    = [\gamma_{i}]\circ[\gamma_{i}]-q_{ii}
    = \Phi(\alpha_{i})\circ\Phi(\alpha_{i})-\Phi(\alpha_{i}*\alpha_{i})
  \end{equation}
  and similarly for~\(i>j\)
  \begin{equation}
    \label{eq:d-gamma-ij-h}
    d[\gamma_{ij}]
    = [\gamma_{i}]\circ[\gamma_{j}]+[\gamma_{j}]\circ[\gamma_{i}]-q_{ii}
    = \Phi(\alpha_{i})\circ\Phi(\alpha_{j})-\Phi(\alpha_{i}*\alpha_{j}),
  \end{equation}
  compare the formula~\eqref{eq:example-mult}.
\end{proof}

\begin{corollary}
  \label{thm:free-general}
  If \(\KK\) acts \(\kk\)-freely on~\(\ZZ_{\Sigma}\),
  then there is an isomorphism of graded algebras
  \begin{equation*}
    H^{*}(\ZZ_{\Sigma}/\KK) \cong \Tor_{H^{*}(BL)}(\kk,\kk[\Sigma])
  \end{equation*}
  with the twisted \(*\)-product on the right-hand side.
  This isomorphism is natural with respect to vertex-preserving morphisms of simplicial posets.
\end{corollary}

\begin{proof}
  This follows by combining \Cref{thm:iso-tor-general} with \Cref{thm:free-quotient}.
\end{proof}

\begin{remark}
  Assume that \(K=1\), that is, \(T=L\).
  Then we can choose the basis~\((\alpha_{i})\) for~\(H^{1}(L;\Z)\) to be dual
  to the basis~\(( x_{v})\) for~\(H_{1}(T;\Z)\),
  so that all coordinates~\(\coord{i}{ x_{v}}\) are either \(0\) or~\(1\).
  From~\eqref{eq:twisting-general} we see that all twisting terms~\(q_{ij}\) vanish
  in this case. We thus recover the algebra isomorphism from \Cref{thm:dga-model-ZZ}.
\end{remark}

\section{The case where \texorpdfstring{\(2\)}{2} is invertible}
\label{sec:2-inv}

In this section we assume that \(2\) is a unit in~\(\kk\).

The inversion map on~\(T\) sends each \(1\)-simplex~\(\barsim{y}\in T\) to~\(\barsim{-y}\).
We can replace each representative~\( c_{v}=\barsim{y_{v}}\)
of the canonical basis~\((y_{v})_{v\in V}\) for~\(H_{1}(T;\Z)\) by
\begin{equation}
  \label{eq:def-tilde-c}
  \tildecc_{v} = \textstyle \frac{1}{2}\,\barsim{y_{v}}-\frac{1}{2}\,\barsim{-y_{v}}.
\end{equation}
This gives a new quasi-iso\-mor\-phism~\(\tilde f_{\Sigma}\colon C^{*}(DJ_{\Sigma})\to\kk[\Sigma]\)
inducing a new quasi-iso\-mor\-phism~\(\tildePsiSigma\) in \Cref{thm:HKZZSigma-Tor}.

We can similarly pass to slightly modified
functions~\(\tilde\alpha_{ij}\colon H_{1}(L;\Z)\to\kk\) with
\begin{equation}
  \label{eq:def-gamma-ij-2}
  \tilde\alpha_{ij} =
  \begin{cases}
    \alpha_{ii}+\frac{1}{2}\alpha_{i} & \text{if~\(i=j\),} \\
    \alpha_{ij} & \text{if~\(i>j\),}
  \end{cases}
\end{equation}
that is,
\begin{equation}
  \tilde\alpha_{ij}(x) =
  \begin{cases}
    \frac{1}{2}\coord{i}{x}\coord{i}{x} & \text{if~\(i=j\),} \\
    \coord{i}{x}\coord{j}{x} & \text{if~\(i>j\)}
  \end{cases}
\end{equation}
for~\(1\le j\le i\le n\) and~\(x\in H_{1}(L;\Z)\).

This leads to new cochains~\(\tilde\gamma_{ij}\in C^{2}(BL)\),
and the identities~\eqref{eq:d-gamma-ij} continue to hold for them.
We stress that we do not (and cannot)
change the representatives~\(\gamma_{i}\).
As before, we define \(\tilde q_{ij}\) to be the image of~\(\tilde\gamma_{ij}\)
in~\(\kk[\Sigma]\) with respect to the map~\(\tilde f_{\Sigma}\).
The following is a special case of~\citegersten{Addendum~\refgersten{thm:cuptwo-DJ}}.

\begin{lemma}
  \label{thm:tilde-qij-0}
  The twisting terms~\(\tildeq_{ij}\) vanish for all~\(i\ge j\).
\end{lemma}

\begin{proof}
  Let~\(i>j\) and~\(v\in V\).
  We use \Cref{thm:value-fT}\,\ref{thm:value-fT-2}
  to determine the coefficient in front of the generator~\(t_{v}\in\kk[V]\).
  We have
  \begin{equation}
    \tilde\alpha_{ij}(-x_{v})
    = (\coord{i}{-x_{v}})(\coord{j}{-x_{v}})
    = \coord{i}{x_{v}}\coord{j}{x_{v}}
    = \tilde\alpha_{ij}(x_{v}),
  \end{equation}
  which together with~\eqref{eq:def-tilde-c}
  implies that the coefficient vanishes and therefore also~\(\tildeq_{ij}\).
  The case~\(i=j\) is analogous.
\end{proof}

\begin{theorem}
  \label{thm:iso-tor-2}
  There is an isomorphism of graded \(\kk\)-algebras
  \begin{equation*}
    H_{\KK}^{*}(\ZZ_{\Sigma}) \cong \Tor_{H^{*}(BL)}(\kk,\kk[\Sigma]),
  \end{equation*}
  where the product on the \(\Tor\)~term is the canonical one.
  The isomorphism is natural with respect to vertex-preserving morphisms of simplicial posets.
\end{theorem}

\begin{proof}
  Combine \Cref{thm:tilde-qij-0} with \Cref{thm:iso-tor-general},
  using the maps~\(\tilde f_{\Sigma}\) and~\(\tildePsiSigma\).
\end{proof}

\begin{corollary}
  \label{thm:free-2}
  Assume that \(\KK\) acts \(\kk\)-freely on~\(\ZZ_{\Sigma}\).
  Then there is an isomorphism of graded \(\kk\)-modules
  \begin{equation*}
    H^{*}(\ZZ_{\Sigma}/\KK) \cong \Tor_{H^{*}(BL)}(\kk,\kk[\Sigma])
  \end{equation*}
  where the product on the \(\Tor\)~term is the canonical one.
  The isomorphism is  
  natural with respect to vertex-preserving morphisms of simplicial posets.
\end{corollary}

\section{Examples}
\label{sec:examples}

\def\bb{b}
\def\Ztwo{\Z_{\bb}}

In this section we illustrate Theorems~\ref{thm:iso-tor-general} and~\ref{thm:iso-tor-2},
and we show that even for quotients by subtori a multiplicative isomorphism
of the form
\begin{equation}
  \label{eq:examples:iso-tor}
  H^{*}(\XXX_{\Sigma}) \cong \Tor_{H^{*}(BL)}(\kk,\kk[\Sigma])
\end{equation}
does not exist in general. As indicated by the notation,
we will phrase our examples in the language of toric varieties.

We start by generalizing \Cref{ex:intro}. We consider the boundary fan~\(\Sigma\)
of a simplicial cone in~\(\R^{n}\) with rays given by the columns of the characteristic matrix
\begin{equation}
  \label{eq:char-matrix-RP}
  \Lambda =
  \begin{bmatrix}
    \bb \\
    -1 & 1 \\
    \vdots & & \ddots \\
    -1 & & & 1
  \end{bmatrix}
  \in\Z^{n\times n}
\end{equation}
for~\(n\ge2\) and \(\bb\)~prime.
The cokernel of the corresponding linear map is isomorphic to~\(\Z_{\bb}\),
and the image of the vector~\([\frac{1}{\bb},\dots,\frac{1}{\bb}]\) is integral.
This shows that the associated smooth toric variety~\(\XXX_{n}\coloneqq\XXX_{\Sigma}\)
is the quotient of~\(\C^{n}\setminus\{0\}\)
by the diagonal action of the group~\(\Gamma_{\bb}\) of \(\bb\)-th roots of unity.
Hence \(\XXX_{n}\) is homotopic to the generalized lens space~\(S^{2n-1}/\Gamma_{\bb}\),
which is also the partial quotient~\(\XX_{\Sigma}\)
since \(\ZZ_{\Sigma}\) is the boundary of~\((D^{2})^{n}\approx D^{2n}\).
In particular, \(\XXX_{n}\simeq\RP^{2n-1}\) for~\(\bb=2\).

\begin{example}
  \label{ex:RP}
  For~\(\kk=\Z_{\bb}\), the torsion product for~\(\XXX_{n}\) looks as follows.
  \begin{equation}
    \label{ex:tor-RP-Ztwo}
    \begin{array}{cc|l}
      \Ztwo & & 2n \\
      \Ztwo & \Ztwo & 2n\mathrlap{{}-2} \\
      \vdots & \vdots & \vdots \\
      \Ztwo & \Ztwo & 4 \\
      \Ztwo & \Ztwo & 2 \\
      & \Ztwo & 0 \\
      \hline
      -1 & 0
    \end{array}
  \end{equation}
  We have
  \begin{equation}
    d\,\alpha_{1} = -\bb\,t_{1}
    \qquad\text{and}\qquad
    d\,\alpha_{k} = t_{1}-t_{k} \quad\text{for~\(k>1\).}
  \end{equation}
  (Here we are using the obvious notation~\(t_{1}\),~\dots~\(t_{n}\) for the generators of~\(\kk[V]\).)
  The generator in bidegree~\((0,2k)\) is~\(t_{1}^{k}\) for~\(0\le k<n\),
  and the one in bidegree~\((-1,2k)\) is \(\alpha_{1}\,t_{1}^{k-1}\) for~\(0<k\le n\).
  As it should be, the power
  \begin{equation}
    t_{1}^{n} = d\,\sum_{k=2}^{n}\alpha_{k}\,t_{k-1}\cdots t_{2}\,t_{1}^{n-k+1}
  \end{equation}
  vanishes in cohomology since~\(t_{1}\cdots t_{n}=0\).
  According to \Cref{thm:iso-tor-general} the product is given by
  \begin{equation}
    \alpha_{1}*\alpha_{1} = q_{11} = {\textstyle \frac{\bb(\bb-1)}{2}}\,t_{1}.
  \end{equation}
  
  For~\(\bb=2\), we have \(q_{11}=t_{1}\),
  which indeed describes the cohomology ring of real projective space.
  As pointed out in \Cref{ex:intro},
  an isomorphism of the form~\eqref{eq:examples:iso-tor}
  would imply that the cup product respects bidegrees,
  hence that all products among elements of odd degree vanish in this example.

  For odd~\(\bb\), the quotient~\((\bb-1)/2\) is an integer. Hence
  \(\alpha_{1}*\alpha_{1}\) is a multiple of~\(\bb\) and therefore vanishes.
  (The class represented by~\(\alpha_{1}\) necessarily squares to~\(0\)
  since \(2\) is invertible in~\(\kk\).)
  Thus, all twisting terms vanish in cohomology, which retrieves
  the cohomology ring of the generalized lens space and confirms \Cref{thm:iso-tor-general}.
\end{example}

\begin{example}
  For~\(\bb=2\) and~\(\kk=\Z\), the torsion product for~\(\XXX_{n}\)
  is given by
  \begin{equation}
    \label{ex:tor-RP-Z}
    \begin{array}{cc|l}
      \Z & & 2n \\
      & \Z_{2} & 2n\mathrlap{{}-2} \\
      & \vdots & \vdots \\
      & \Z_{2} & 2 \\
      & \Z & 0 \\
      \hline
      -1 & 0
    \end{array}
  \end{equation}
  with the same generators in bidegrees~\((0,2k)\) and~\((-1,2n)\) as before.

  The product~\(\XXX_{m}\times\XXX_{n}\) is given by the product of the fans
  (and the corresponding moment-angle complexes by the join of the simplicial complexes).
  According to the Künneth formula, the cohomology~\(H^{*}(\XXX_{m}\times\XXX_{n})\) with~\(m\),~\(n>2\)
  has additional torsion terms in its cohomology produced by~\(\Tor(\Z_{2},\Z_{2})=\Z_{2}\).
  For instance, for~\(m=2\) and~\(n=3\) we obtain
  \begin{equation}
    \setlength{\fboxsep}{1pt}
    \begin{array}{ccc|c}
      \Z &  &  & 10 \\
       & \Z_{2}\oplus\Z_{2} &  & 8 \\
       & \Z\oplus\Z_{2}\oplus\fbox{\(\Z_{2}\)} & \Z_{2} & 6 \\
       & \Z\oplus\fbox{\(\Z_{2}\)} & \Z_{2}\oplus\Z_{2} & 4 \\
       &  & \Z_{2}\oplus\Z_{2} & 2 \\
       \phantom{\Z_{2}\oplus\Z_{2}} &  & \Z & 0 \\
      \hline
      -2 & -1 & 0 & \mathrlap{\;\;\;\; ,}
    \end{array}
  \end{equation}
  where the torsion terms are framed.
  (Recall that the differential goes horizontally
  if we arrange the terms in the Koszul complexes so as to give our tables.)

  We write the relevant variables corresponding to~\(\XXX_{n}\) as~\(\alpha_{1}\) and~\(t_{1}\)
  and the ones corresponding to~\(\XXX_{m}\) as~\(\beta_{1}\) and~\(s_{1}\).
  For any~\(0<k<n\) and~\(0<l<m\) we get a torsion term~\(\Z_{2}\) in bidegree~\((-1,2(k+l))\)
  with generator
  \begin{equation}
    \label{eq:def-xkl}
    x_{kl} = \alpha_{1}\,t_{1}^{k-1}\,s_{1}^{l} - \beta_{1}\,t_{1}^{k}\,s_{1}^{l-1}.
  \end{equation}
  As before, we have the twisted products~\(\alpha_{1}^{2}=t_{1}\) and~\(\beta_{1}^{2}=s_{1}\).
  The square of~\(x_{kl}\) therefore is given by
  \begin{equation}
    x_{kl}^{2} = t_{1}^{2k-1}\,s_{1}^{2l} + t_{1}^{2k}\,s_{1}^{2l-1} \,,
  \end{equation}
  which is not a coboundary for~\(m\),~\(n\) large enough.
  However, an isomorphism of the form~\eqref{eq:examples:iso-tor} would
  force all these squares to vanish.
  The smallest counterexample is
  the one illustrated in the table above,
  where we get~\(x_{11}^{2}=t_{1}s_{1}^{2}+t_{1}^{2}s_{1}=t_{1}^{2}s_{1}\),
  which is the generator in bidegree~\((0,6)\).
\end{example}

So far, our examples have been quotients
by the disconnected groups~\(\Gamma_{\bb}\) and \(\Gamma_{2}\times \Gamma_{2}\).
As remarked above, this is caused by the fact
that the columns in the characteristic matrix~\eqref{eq:char-matrix-RP}
do not span the lattice~\(\Z^{n}\). To fix this, we simply add another column, say the first canonical
basis vector~\(e_{1}\). Since \(e_{1}\) does not lie in the support of the fan~\(\Sigma\) considered so far,
we obtain a new fan~\(\Sigma'\) as the union of~\(\Sigma\) and the ray through~\(e_{1}\),
hence a new smooth toric variety~\(\XXX_{n}'=\XXX_{\Sigma'}\).
It is now the quotient of its Cox construction by a free action of~\(\C^{\times}\).

The toric variety given by a single ray in~\(\R^{n}\) is isomorphic to~\(\C\times(\C^{\times})^{n-1}\),
with torsion product
\begin{equation}
  \begin{array}{cccc|l}
    \kk^{(n\mathrlap{-1)}} & & & & 2n\mathrlap{{}-2} \\
    & \ddots & & & \vdots \\
    & & \kk^{(1)} & & 2 \\
    & & & \kk^{(0)} & 0 \\
    \hline
    -(n-1) & \cdots & \mathllap{-}1 & 0
  \end{array}
\end{equation}
for any~\(\kk\). Here we have written \((k)\) for the binomial coefficient~\(\binom{n-1}{k}\).

By a Mayer--Vietoris argument, we see that the torsion product for~\(\XXX_{n}'\)
differs from the one for~\(\XXX_{n}\) in that the terms in bidegree~\((0,2)\)
and~\((-1,2)\) are removed and
terms~\(\kk^{(k)}\) added in bidegree~\((-k,2k+2)\) for \(0\le k<n\).

\goodbreak

\def\Ztwo{\Z_{2}}

\begin{example}
  Assume again \(\bb=2\) and~\(\kk=\Z_{2}\).
  What we have said so far implies
  that \Cref{ex:RP} still works for~\(\XXX_{n}'\)
  if we look at~\(n=4\) and \(x=\alpha_{1}\,t_{1}\),
  for instance. The torsion product is as follows:
  \begin{equation}
    \begin{array}{cccc|l}
      \Ztwo & & \Ztwo & & 8 \\
      & \Ztwo^{3} & \Ztwo & \Ztwo & 6 \\
      & & \Ztwo^{4} & \Ztwo & 4 \\
      & & & \Ztwo & 2 \\
      & & & \Ztwo & 0 \\
      \hline
      -3 & -2 & -1 & 0
    \end{array}
  \end{equation}
  In particular, it vanishes in bidegree~\((-2,8)\).
  Since \((-1,4)\) is the only non-zero bidegree in degree~\(3\),
  the Buchstaber--Panov formula predicts that the square of any degree~\(3\) elements vanishes.
  However, \(x^{2}=t_{1}^{3}\) is the generator in bidegree~\((0,6)\)
  which follows from the same calculation as before.
  Alternatively, one can look at the restriction map~\(H^{*}(\XXX_{n}')\to H^{*}(\XXX_{n})\cong H^{*}(\RP^{7})\),
  which is surjective in degrees~\(\ne1\).
\end{example}

\begin{example}
  For~\(\bb=2\) and~\(\kk=\Z\),
  the element~\(x_{kl}\) defined in~\eqref{eq:def-xkl}
  still generates a torsion term~\(\Z_{2}\) in the cohomology of~\(\XXX_{m}'\times\XXX_{n}'\)
  for~\(k\),~\(l\ge2\),
  and its square represents a non-zero class in bidegree \((0,4(k+l)-2)\)
  for~\(m\),~\(n\) large enough. Once again, the canonical multiplication
  in~\eqref{eq:examples:iso-tor} fails to describe this product.
\end{example}

\section{Concluding remarks}

In this final section we again write topological spaces in the form~\(X\),~\(Y\),~\(Z\).

\subsection{Real and rational coefficients}
\label{sec:realcoeffs}

We had to face two main difficulties when proving our main theorems:
The cup product of singular cochains is not graded commutative,
and the quasi-iso\-mor\-phism~\(C^{*}(BT)\to H^{*}(BT)\) depends on
choices and is not natural with respect to arbitrary maps between tori.
Both problems disappear if one uses real or rational coefficients and (polynomial) differential forms.

Let us first consider the case of a smooth toric variety~\(\XXX_{\Sigma}\) and \(\kk=\R\).
We write \(T_{\sigma}\) for the subtorus of~\(T\)
corresponding to the linear hull of~\(\sigma\in\Sigma\) in~\(\R^{n}=H_{1}(T)\).
The inclusion of invariant differential forms gives
a quasi-iso\-mor\-phism of cdgas
\( 
  H^{*}(T) \hookrightarrow \Omega^{*}(T)
\), 
natural in~\(T\).
From this we get a zigzag of cdga quasi-iso\-mor\-phisms
\begin{equation}
  \label{eq:zigzag-real}
  \Omega^{*}(\XXX_{\sigma})
  \leftarrow \Omega^{*}(T/T_{\sigma})
  \leftarrow H^{*}(T/T_{\sigma})
  \rightarrow \Kl_{\sigma} = H^{*}(T)\otimes H^{*}(BT_{\sigma}),
\end{equation}
natural with respect to the inclusion of faces of~\(\sigma\in\Sigma\),
compare~\cite[Sec.~3.3]{Franz:2001}.
An argument as in~\cite{Franz:2001} shows that the dgas~%
\(\Omega^{*}(\XXX_{\Sigma})\) and~\(\Kl_{\Sigma}\)
are quasi-iso\-mor\-phic, naturally with respect to
inclusion of subfans. In particular, there is an isomorphism
of graded algebras
\begin{equation}
  H^{*}(\XXX_{\Sigma};\R) = \Tor_{H^{*}(BL;\R)}(\R,\R[\Sigma]).
\end{equation}

In the context of polynomial differential forms
we can obtain a quasi-iso\-mor\-phism of cdgas that generalizes
\cite[Thm.~6.2]{PanovRay:2008} from moment-angle complexes to partial quotients\footnote{%
  Such a quasi-iso\-mor\-phism is also implicit in~\cite[Thms.~8.1~\&~8.39]{BuchstaberPanov:2004},
  where the original argument for the isomorphism~\eqref{eq:intro:E2}
  from~\cite{BuchstaberPanov:2002} was modified from cellular cochains
  to rational coefficients and polynomial differential forms. We note that
  in the proof of~\cite[Thm.~8.39]{BuchstaberPanov:2004} one has to address
  the problem that in general the formality maps for~\(BT^{r}\) and~\(BT^{m}\) are not compatible
  with the morphisms induced by the map~\(BT^{r}\to BT^{m}\).
  This can be solved by using simplicial constructions as in \Cref{sec:tori}
  or by an argument analogous to~\cite[p.~176]{BaumSmith:1967}.}%
\footnote{%
  I have not been able to understand the proofs given in~\cite[Thm.~8.1.6, Lemma~8.1.9]{BuchstaberPanov:2015}.
  They seem to use ordinary differential forms with polynomial coefficients
  instead of polynomial differential forms in the sense of Sullivan and to assume
  a (non-existing) polynomial parametrization of the circle in the plane.}
and also provides a different proof of~\cite[Cor.~7.2]{PanovRay:2008}.

\begin{proposition}
  Assume \(\kk=\Q\).
  \begin{enumroman}
  \item The cdgas~\(\APL^{*}(X_{\Sigma})\) and~\(\Kl_{\Sigma}\)
    are quasi-iso\-mor\-phic for any simplicial poset~\(\Sigma\),
    naturally with respect to vertex-preserving morphisms.
  \item \textup{(Panov--Ray)}
    If \(\Q[\Sigma]\cong\HL^{*}(X_{\Sigma})\) is free over~\(H^{*}(BL)\), then \(X_{\Sigma}\) is formal
    in the sense of rational homotopy theory.
  \end{enumroman}
\end{proposition}

\begin{proof}
  For the first part, we imitate Panov--Ray's reasoning in~\cite[Sec.~6]{PanovRay:2008}.
  Inspection of their proof shows that we only need a quasi-iso\-mor\-phism of cdgas
  \( 
    H^{*}(T) \to \APL^{*}(T)
  \) 
  that is natural with respect to morphisms of tori,
  analogously to the inclusion of invariant differential forms used in~\eqref{eq:zigzag-real}.
  By \Cref{thm:simplicial-torus-inclusion} we can assume that \(T=BN\) is the simplicial torus
  associated to some lattice~\(N\).
  
  Given an element~\(\alpha\in H^{1}(BN)\cong \Hom_{\Z}(N,\Q)\), we define
  \begin{equation}
    \psi(\alpha)_{\aa} = \sum_{k=1}^{n} \alpha(a_{1}+\dots+a_{k})\,dt_{k}
  \end{equation}
  for an \(n\)-simplex~\(\aa=[a_{1}|\dots|a_{n}]\in BN\). (Note the similarity with~\eqref{eq:def-1-psg}.)
  It is straightforward to check that \(\psi(\alpha)\) is a well-defined closed \(1\)-form on~\(BN\)
  and that \(\psi\) extends to a quasi-iso\-mor\-phism~\(H^{*}(BN)\to\APL^{*}(BN)\), natural in~\(N\).

  Now assume that \(\HL^{*}(X_{\Sigma})\) is free over~\(H^{*}(BL)\).
  In this case we have an isomorphism
  \( 
    H^{*}(X_{\Sigma}) \cong \Q[\Sigma]/\mmm
  \), 
  where \(\mmm\lhd H^{*}(BL)\) is the maximal homogeneous ideal,
  \cf~\cite[p.~3]{BifetDeConciniProcesi:1990} or~\cite[Thm.~7.4.35]{BuchstaberPanov:2015}.
  Moreover, the projection map
  \begin{equation}
    \Kl_{\Sigma} \to \Q[\Sigma]\,/\,\mmm,
    \qquad
    \alpha\,f \mapsto [f]
  \end{equation}
  is a quasi-iso\-mor\-phism of cdgas, which proves the second part.
\end{proof}

\subsection{Cohomology of quotient stacks}

For any algebraic subgroup~\(\KKK\subset\TTT\) one can consider
the quotient stack~\([\ZZZ_{\Sigma}/\KKK]\), which is a toric Deligne--Mumford stack
if \(\KKK\) acts with finite isotropy groups.
Edidin~\cite[Thm.~3.16]{Edidin:2013} has shown that
the cohomology of the stack~\([\ZZZ_{\Sigma}/\KKK]\)
is the \(\KKK\)-equivariant cohomology of~\(\ZZZ\).
Hence \Cref{thm:iso-tor-general}
can be read as describing the cohomology rings of these quotient stacks.

\subsection{Real toric varieties and real partial quotients}

We briefly survey the state of affairs for the cohomology rings
of real toric varieties and partial quotients of
real moment-angle complexes~\(\R Z_{\Sigma}=\ZZ_{\Sigma}(D^{1},S^{0})\),
\cf~\cite[Sec.~6.6]{BuchstaberPanov:2002}.

Real moment-angle complexes are homotopy-equivalent to ``real Cox constructions'',
which are complements of real coordinate subspace arrangements.
A description of their integer cohomology ring was given
by de~Longueville~\cite[Thm.~1.1]{DeLongueville:1999} in~1999.
In 2013, however, Gitler and López de Medrano~\cite[Sec.~3]{GitlerLopezDeMedrano:2013}
found an example of a simplicial complex~\(\Sigma\)
such that rings~\(H^{*}(Z_{\Sigma};\Z)\) and~\(H^{*}(\R Z_{\Sigma};\Z)\)
are torsion-free and modulo~\(2\) non-iso\-mor\-phic as ungraded rings,
which contradicts de Longueville.
A corrected product formula was recently provided by Cai~\cite[p.~514]{Cai:2017}.

Based on this, Choi--Park~\cite[Thm.~4.5]{ChoiPark:2017},~\cite[Main Thm.]{ChoiPark:2017a}
have computed the cohomology ring of a partial quotient~\(\R X_{\Sigma}=\R Z_{\Sigma}/K\)
by a freely acting subgroup~\(K\) of~\(G=(\Z_{2})^{V}\)
under the assumption that \(2\in\kk\) is invertible.
Suciu--Trevisan~\cite{Suciu:2013}, \cite[Sec.~4.2]{Trevisan:2012}
had previously determined the rational Betti numbers of small covers.

Assume \(\kk=\Z_{2}\).
By an argument as in \Cref{rem:finitedim-complex},
one can see that in this case Cai's description is additively of the form
\begin{equation}
  \label{eq:iso-Tor-real}
  H^{*}(\R Z_{\Sigma}) = \Tor_{H^{*}(BG)}(\kk,\kk[\Sigma]),
\end{equation}
where the grading of the Stanley--Reisner ring~\(\kk[\Sigma]\)
now is such that linear elements have degree~\(1\),
as do the generators of the polynomial algebra~\(H^{*}(BG)\).
The product on the \(\Tor\)~term is \emph{not} the standard one,
as can be seen for~\(\Sigma=\{\emptyset\}\) already: In this case we have \(\R Z_{\Sigma}=G\) and~\(\kk[\Sigma]=\kk\).
The right-hand side of~\eqref{eq:iso-Tor-real} therefore is a (strictly) exterior algebra,
while \(H^{*}(\R Z_{\Sigma})\) is the algebra of all functions~\(G\to\kk\) with pointwise multiplication,
which are all idempotent. Arguably, this difference to the case of~\(Z_{\Sigma}\)
stems from the very fact that \(H^{*}(G)\) is not an exterior algebra anymore.
If \(H^{*}(G)\) were a primitively generated exterior algebra,
then the calculation of the cohomology ring of a moment-angle complex done in~\cite{Franz:2003a} would
carry over to give not just an additive, but even a multiplicative isomorphism in~\eqref{eq:iso-Tor-real}.

Nothing seems to be known about
the analogous formula for partial quotients,
\begin{equation}
  H^{*}(\R X_{\Sigma}) = \Tor_{H^{*}(BL)}(\kk,\kk[\Sigma]),
\end{equation}
except that by a repeated application of~\cite[Prop.~7.1.6]{Hausmann:2014}
it holds even multiplicatively
if \(\kk[\Sigma]=H_{L}^{*}(\R X_{\Sigma})\) is free over~\(H^{*}(BL)\).
This includes the case of smooth projective real toric varieties
considered by Jurkiewicz~\cite[Thm.~4.3.1]{Jurkiewicz:1985} and
its generalization to small covers
due to Davis--Januszkiewicz~\cite[Thm.~4.14]{DavisJanuszkiewicz:1991}.

\end{document}